\documentclass[12pt,reqno]{article}
 \usepackage[body={6.5in,9.0in},left=1.0in, top=1.0in,centering]{geometry}    
\usepackage{graphicx,color}
\usepackage{amsmath,amsthm,amssymb,amsfonts,multirow,setspace,epsfig,latexsym}
\usepackage[letterpaper=true,colorlinks=true,urlcolor=blue,linkcolor=blue,citecolor=blue,pdfstartview=FitH]{hyperref}
\usepackage{pstricks}
\usepackage{epstopdf}
 \usepackage[sort,comma]{natbib}
\allowdisplaybreaks
\begin{document}

\makeatletter \@addtoreset{equation}{section}
\renewcommand{\thesection}{\arabic{section}}
\renewcommand{\theequation}{\thesection.\arabic{equation}}

\newcommand{\bdd}{\hspace*{-0.08in}{\bf.}\hspace*{0.05in}}
\def\para#1{\vskip 0.4\baselineskip\noindent{\bf #1}}
\def\qed{$\qquad \Box$}
\def\rr{{\mathbb R}}
\def\ss{{\mathbb S}}
\def\zz{{\Bbb Z}}
\def\G{{\mathcal G}}
\def\chao#1 {\fbox {\footnote {\ }}\ \footnotetext {From Chao: #1}}

\def\fubao#1 {\fbox {\footnote {\ }}\ \footnotetext {From Fubao: #1}}

\newcommand{\bed}{\begin{displaymath}}
\newcommand{\eed}{\end{displaymath}}
\newcommand{\bea}{\bed\begin{array}{rl}}
\newcommand{\eea}{\end{array}\eed}
\newcommand{\disp}{\displaystyle}
\newcommand{\ad}{&\!\!\!\disp}
\newcommand{\aad}{&\disp}
\newcommand{\barray}{\begin{array}{ll}}
\newcommand{\earray}{\end{array}}
\newcommand{\La}{\Lambda}
\newcommand{\la}{\lambda}
\newcommand{\sg}{\sigma}
\newcommand{\lf}{\lfloor}
\newcommand{\rf}{\rfloor}
\newcommand{\wdt}{\widetilde}
\newcommand{\wdh}{\widehat}
\newcommand{\bQ}{{\mathbb Q}}
\newcommand{\al}{\alpha}
\newcommand{\lbar}{\overline}
\newcommand{\tr}{\mathrm {tr}}
\newcommand{\sgla}{\sigma_{\lambda_{0}}}
\newcommand{\lan}{\langle}
\newcommand{\ran}{\rangle}

\newcommand{\im}{\mathtt{i}}

\newcommand{\E}{{\mathbb E}}
\newcommand{\bP}{{\mathbb P}}
\newcommand{\R}{\mathbb R}
\renewcommand{\P}{\mathbb P}
\newcommand{\Q}{\mathbb Q}
\newcommand{\LL}{{\mathcal L}}
\newcommand{\A}{{\mathcal A}}
\newcommand{\B}{{\mathcal B}}
\newcommand{\F}{{\mathcal F}}
\newcommand{\N}{{\mathcal N}}
\newcommand{\e}{\varepsilon}
\newcommand{\one}{\mathbf 1}
\renewcommand{\d}{\mathrm{d}}
\newcommand{\set}[1]{\left\{#1\right\}}
\newcommand{\abs}[1]{\left\vert#1\right\vert}
\newcommand{\norm}[1]{\left\|#1\right\|}

\newtheorem{Theorem}{Theorem}[section]
\newtheorem{Corollary}[Theorem]{Corollary}
\newtheorem{Proposition}[Theorem]{Proposition}
\newtheorem{Lemma}[Theorem]{Lemma}
\theoremstyle{definition}
\newtheorem{Definition}[Theorem]{Definition}

\newtheorem{Note}[Theorem]{Note}

\newtheorem{Remark}[Theorem]{Remark}
\newtheorem{Example}[Theorem]{Example}
\newtheorem{Counterexample}[Theorem]{Counterexample}
\newtheorem{Assumption}[Theorem]{Assumption}

\makeatletter
\renewcommand\theequation{\thesection.\@arabic\c@equation}
\makeatother
\newcommand{\beq}[1]{\begin{equation} \label{#1}}
\newcommand{\eeq}{\end{equation}}

\title{On the Martingale Problem and Feller and Strong Feller Properties for Weakly Coupled L\'evy Type Operators }
\author{Fubao Xi\thanks{School of Mathematics and Statistics,
Beijing Institute of Technology, Beijing 100081, China, xifb@bit.edu.cn.}
\and Chao Zhu\thanks{Department of Mathematical Sciences, University of Wisconsin-Milwaukee, Milwaukee, WI
53201, zhu@uwm.edu.}}

\maketitle

\begin{abstract}
This paper considers the martingale problem for a class of  weakly coupled  L\'{e}vy type operators. It is shown that under some mild conditions, the martingale problem is well-posed and  uniquely determines a strong Markov process $(X,\Lambda)$. The process $(X,\Lambda)$, called a regime-switching jump diffusion with L\'evy type jumps,  is further shown to posses Feller and strong Feller properties  under non-Lipschitz conditions via the coupling method.

\bigskip


\bigskip


\bigskip

\noindent{\bf 2000 MR Subject Classification.} 60J25, 60J27, 60J60,
60J75.
\end{abstract}

\renewcommand{\baselinestretch}{1.18}

\section{Introduction}\label{I}

This paper deals with the martingale problem for a weakly coupled L\'evy type operator $\A$ defined as follows.
Let $d$ and $n_{0}$ be two positive integers and set $\ss :=\{1, 2, \cdots, n_{0}\}$.
For all  ``nice'' functions $f$: $\R^{d}\times
\ss \to \R$, we  define
\begin{equation}\label{A}
\A f(x,k):=\LL_{k}f(x,k)+Q(x)f(x,k).
\end{equation}
Here, for each $k \in \ss$, ${\LL}_{k}$ is
a L\'{e}vy type operator defined as follows:
\begin{equation}\label{L}
\begin{array}{ll} {\LL}_{k} f(x,k)\ad
: =\frac {1}{2}\hbox{tr}\bigl(a(x,k)\nabla^2 f(x,k)\bigr)+\langle
b(x,k),
\nabla f(x,k)\rangle\\
\aad \quad + \int_{\R^{d}_{0}}
\bigl(f(x+u,k)-f(x,k)- \langle \nabla f(x,k), u\rangle
{\mathbf{1}}_{B(0, \e_{0})}(u)\bigr)\nu(x,k,\d u),
\end{array}
\end{equation} where for each $(x,k) \in \rr^{d} \times \ss$,
$a(x,k)=\bigl(a_{ij}(x,k)\bigr)\in \R^{d\times d}$ is  symmetric and nonnegative definite,  $b(x,k)=\bigl(b_{i}(x,k)\bigr) \in \R^{d}$, and $\nu (x,k,\cdot)$ is a L\'{e}vy kernel
 such that for each $(x,k)$,  $\nu (x,k, \cdot)$ is a nonnegative $\sigma$-finite
measure on $\R^{d}_{0}$ satisfying   
\begin{equation}
\label{eq-Levy-measure}
\int_{\R^{d}_{0}}|u|^{2}{\mathbf{1}}_{B(0, \e_{0})}(u)\nu (x,k, \d u)<\infty \text{ and }\nu
(x,k,\rr^{d}\setminus  B(0, \varepsilon_{0}))<\infty,
\end{equation} where $\e_{0}>  0$ (one can
usually take $\varepsilon_{0}=1$).
 Here and hereafter, $\nabla f(\cdot, k)$ and $\nabla^{2} f(\cdot,k)$
denote respectively the gradient and Hessian matrix of
$f(\cdot,k)$, $\langle \cdot , \cdot \rangle$ denotes the
inner product in $\R^{d}$, $\R^{d}_{0}: =\R^{d}\setminus \{0\}$, and $B(0, r):=\{x\in\R^{d}: |x|<r\}$ for $r>0$.
 In \eqref{A} and throughout the paper, the switching operator $Q(x)$ is defined as follows:
\begin{equation}\label{Q(x)}
Q(x)f(x,k): =\sum_{l \in \ss} q_{kl}(x) \bigl(f(x,l)-f(x,k)\bigr),
\end{equation}
where $Q(x)=\bigl(q_{kl}(x)\bigr)$ is an $n_{0} \times n_{0}$ matrix-valued
measurable function on $\rr^{d}$ such that for all $  x \in  \rr^{d}$  we have
$ q_{kl}(x) \ge 0 $       for  $k \ne l$,
and for each  $k \in \ss$,
$\sum_{l \in \ss} q_{kl}(x) =0$.

In this paper, we  consider the martingale problem for the
weakly coupled L\'{e}vy type operator $\A$ defined in \eqref{A}  on
$\Omega:=D([0,\infty),\rr^{d} \times \ss)$,
  the space of right continuous functions on $[0,\infty)$ into $\rr^{d} \times \ss$ having left limits endowed with the Skorohod topology.
 Let ${\F}_{t}$ be the $\sg$-field generated by the cylindrical sets on $D([0,\infty),\rr^{d} \times \ss)$ up to time $t$ and   set
${\F}=\bigvee_{t=0}^{\infty} {\F}_{t}$.  Next, let $C^{\infty}_{c}(\rr^{d}\times \ss)$ denote the family of   functions defined on $\rr^{d}\times \ss$
such that $f(\cdot,k) \in C^{\infty}_{c}(\rr^{d})$ with $k \in \ss$, where $C^{\infty}_{c}(\rr^{d})$ denotes the family of
functions defined on $\rr^{d}$ which are infinitely   differentiable and have compact supports.

\begin{Definition} \label{defn-mg-A} For a given $(x,k)\in
\rr^{d} \times \ss$, we say a probability measure $\bP^{(x,k)}$ on
$D([0,\infty),\rr^{d} \times \ss)$ is a solution to the martingale
problem for the operator $\A$ starting from $(x,k)$, if
$\bP^{(x,k)}((X(0),\La(0))=(x,k))=1$ and for each function $f \in
C^{\infty}_{c}(\rr^{d}\times \ss)$, \begin{equation}\label{martingale1}
M_{t}^{(f)}:=f(X(t),\La(t))-f(X(0),\La(0))-\int_{0}^{t} \A f(X(s),\La(s))\d s \end{equation}
is an $\{\F_{t}\}$-martingale with respect to $\bP^{(x,k)}$, where $(X,\La)$ is the coordinate process defined by $(X(t,\omega), \La(t, \omega)) = \omega(t)\in \R^{d}\times \ss$ for all $t \ge 0$ and $\omega\in \Omega$.

Sometimes, we say
that the probability measure $\bP^{(x,k)}$ is a martingale solution
for the operator $\A$ starting from $(x,k)$. We often
call  the coordinate process $(X,\La)$ the {\em regime-switching jump diffusion with L\'{e}vy
type jumps}.
\end{Definition}

Since the seminal work of Stroock and Varadhan (\cite{StroockV-691,StroockV-69II}) on  martingale problems for second order diffusion operators, the notion of martingale problems have been extensively studied for various processes  in the literature. For example,  \cite{Koma-73} and \cite{Stroock-75} prove that the martingale problem for a L\'{e}vy  type operator is well-posed; \cite{Bass-88} investigates the martingale problem for pure jump Markov processes; \cite{DawsonZ-91} and \cite{FengZ-92} considers the martingale problem for a class of nonlinear master equations for chemical reaction models; \cite{Xi-98} and \cite{ZhengZ-86} discuss the martingale problem for $Q$-processes;   \cite{Zamb-00} provides an analytic approach for existence and uniqueness for martingale problems in infinite dimensions; \cite{Kurtz-98} presents a martingale problems for conditional distributions of Markov processes; \cite{MikuR-99} studies martingale problems for stochastic partial differential equations; \cite{Perkins-95} investigates  the martingale problem for interactive measure-valued branching diffusions; \cite{Hoh-94} investigates the martingale problems for psudo-differential operators; and \cite{BassT-09} is devoted to the martingale problem for  stable-like processes.

This paper is motivated by \cite{Stroock-75} and considers weakly coupled L\'{e}vy  type operator $\A$ defined in \eqref{A}.  Roughly speaking, in addition to the diffusion term, the drift term,  and the jump term spelled out in \eqref{L} for each $k\in \ss$, $\A$ also contains a  component $Q(x)$ defined in \eqref{Q(x)}, which provides the switching mechanism for the operators $\LL_{k}, k\in \ss$. In other words, the operators $\LL_{k}, k\in \ss$ are coupled through the operator $Q(x)$ of \eqref{Q(x)}. Therefore it is convenient to call the operator $\A$ of \eqref{A} a weakly coupled L\'{e}vy  type operator and the coordinate process $(X,\La)$ a   
{\em regime-switching L\'evy type process}.
Here we remark that $Q(x)= (q_{kl}(x))$ depends on $x$. When the L\'evy kernel $\nu(x,k,\d z)$ is independent of $(x,k)$, then $\A$ reduces to the infinitesimal generator of a regime-switching jump diffusion process as those considered in \cite{Xi-09,Yin-Xi-10,ZhuYB-15}.  Thanks to their ability in incorporating both  structural changes and jumps of various sizes, regime-switching (jump) diffusion processes   have attracted many interests lately. See, for example,
 \cite{Yin-Xi-10,Xi-09,Xi-08-Feller,SethiZ, YinZh,MaoY,YZ-10,ZhuYB-15,ShaoX-14,Wang-14,XiZ-06,CloezH-15} and   references therein for investigations  of such processes  and their applications in areas such as inventory control,   ecosystem modeling, manufacturing and production planning, financial engineering,    risk theory,   etc.

 However, we notice that in these papers, the jump mechanism is usually assumed to be a finite  or a L\'evy measure $\nu(\d z)$. The study of regime-switching jump diffusions with L\'evy type jumps is relatively scarce, which is precisely the focus of this paper.   In addition, in leu of the stochastic differential equation approach in the aforementioned papers, this  paper begins with the martingale problem for the weakly coupled L\'evy type operator $\A$ of \eqref{A}. We prove that under very mild conditions, the martingale problem for  the operator $\A$ is well-posed. That is, we show that for any $(x,k) \in \R^{d}\times \ss$, there is exactly one martingale solution for the operator $\A$ starting from $(x,k)$. This is achieved in two steps.
 In the first step, we assume that    $Q$ of \eqref{Q(x)} takes a special form ($\wdh Q$ in \eqref{eq-Q-hat}); consequently  $\A$ of \eqref{A}  reduces to   $\wdh \A $ of \eqref{eq-A-hat}. For such a special operator $\wdh \A$,   under Assumption \ref{assump-Lhasonlyone},
 we   manipulate  the Stroock-Varadhan piecing together
 method (refer to \S 6.1 of \cite{Stroock-V}) to construct
 a martingale solution for the operator $\widehat {\mathcal A}$
 with an arbitrary initial condition $(x,k)\in \R^{d}\times\mathbb S$
 and further show that this solution    is weakly unique
 in Theorem \ref{special}.
 The second step deals with the general case when $Q(x)$ of \eqref{Q(x)}  is $x$-dependent. For such a case, we utilize the likely ratio martingale $M$ defined in \eqref{Mt} to establish the desired existence and uniqueness result in Theorem \ref{thm-general}. One of the key steps in this approach is to show that the switching times and the jump times are mutually disjoint with probability one; see Proposition \ref{prop-no-common-jump} for details.  Such a strategy of using the likelihood ratio martingale  was used in the recent paper  \cite{Xi-09}, where the jump component  is driven by a finite measure. In this paper, we  develop this approach to handle the general weakly coupled  L\'evy type operator $\A$.

 Having established that the martingale problem for $\A$ is well posed, we then have determined a strong Markov process $(X,\La)$ with state space $\R^{d}\times \ss$. The second part of this paper proves that such a process possesses the Feller and strong Feller properties. Here the main tool is the coupling method. For the introduction to coupling method and its applications in various areas of probability and stochastic analysis, we refer to \cite{Chen04,LindR-86,Lindvall,HairMS-11,PriolaW-06,Wang-10} and the references therein.   In this paper,  we first use the coupling method to show that for each $k\in \ss$,    the  process $\wdt X^{(k)}$ corresponding to the L\'evy type operator $\LL_{k}$ of \eqref{L} is Feller under Assumption \ref{FP1}, in which the coefficients, and in particular, the L\'evy type kernel  of the operator  $\LL_{k}$,  are  {\em non-Lipschitz} in the $x$ variable.
 In order to establish the Feller property for the process $(X,\La)$, we kill the L\'evy type process  $\wdt X^{(k)}$ at rate $-q_{kk}$ to obtain the  process  $ X^{(k)}$; see \eqref{kp1} for details.
A mild condition on the functions $q_{kl}(x)$ (Assumption \ref{FP4}) then helps us to derive the Feller property for the killed L\'evy type process $X^{(k)}$ in Lemma \ref{FP5}.  Finally we use a series representation for the resolvent $G_{\alpha}$ of the process $(X,\La)$ and a result in \cite{MeynT-93} to establish the Feller property for  the process $(X,\La)$; this is spelled out in Theorem \ref{thm-Feller}.

 Next we use a similar approach to establish the strong Feller property for the process  $(X,\La)$ in Section \ref{sect-str-Feller}. More precisely, inspired by \cite{PriolaW-06}, we use a combination of reflection and marching coupling for the operator $\LL_{k}$ to establish the strong Feller property   for the  processes  $\wdt X^{(k)}$ and   $ X^{(k)}$ in Proposition \ref{prop-str-Fe}.  Again, we allow the coefficients and the L\'evy type kernel  of the operator  $\LL_{k}$ to be  {\em non-Lipschitz} in the $x$ variable in Proposition \ref{prop-str-Fe}. Then, as in Section \ref{sect-Feller}, the series representation for the resolvent $G_{\alpha}$ of the process $(X,\La)$ and the aforementioned  result in \cite{MeynT-93}  lead to the desired strong Feller property for the process $(X,\La)$ in Theorem \ref{thm-str-Feller}.

The rest of the paper is arranged as follows. We present the necessary assumptions as well as some preliminary results  in Section \ref{sect-preliminaries}. In addition, Section \ref{sect-preliminaries}   presents some martingales associated with the operator $\A$ (Theorem \ref{thm-martingales}). These martingales  are interesting in their own rights. Moreover, they are useful in  the proofs of Section \ref{sect-general}.  The well-posedness of the martingale problem for $\A$ is divided into two parts: Section   \ref{sect-special} treats the special case when $\A$ is given by $\wdh \A$ of \eqref{eq-A-hat} and Section \ref{sect-general} deals with the general case. Section \ref{sect-Feller} is devoted to proving  the Feller property for the process $(X,\La)$.  Strong Feller property is established in Section \ref{sect-str-Feller}.

To facilitate later presentations, let us introduce some notations that will be frequently used throughout the paper.
  Let $D([0,\infty),\rr^{d})$ (resp., $D([0,\infty),\ss)$) be the space of right continuous  functions on $[0,\infty)$ into $\rr^{d}$ (resp., $\ss$) having left limits endowed with the Skorohod topology, and let ${\G}_{t}$ (resp., ${\N}_{t}$) be the $\sg$-field generated by the cylindrical sets on $D([0,\infty),\rr^{d})$ (resp., $D([0,\infty),\ss)$) up to time $t$. Also denote ${\G}=\bigvee_{t=0}^{\infty} {\G}_{t}$ and ${\N}=\bigvee_{t=0}^{\infty} {\N}_{t}$. It is easy to see that ${\F}_{t}={\G}_{t}\bigvee {\N}_{t}$ for any $t\ge 0$ and that $\F = \G\bigvee \N$.
  Let $C^{2} (\R^{d} \times \ss)$ be the family of functions defined on
  $\R^{d}\times \ss$ such that $f(\cdot, k)\in C^{2}(\R^{d})$
  for each $k \in \ss$ and let $C^{2}_{b} (\R^{d} \times \ss)$
  be the family of bounded functions defined on $\R^{d}\times \ss$
  such that $f(\cdot, k)\in C^{2}(\R^{d})$
  with bounded first and second order continuous partial derivatives in $x$
  for each $k \in \ss$.
Moreover, we denote
    by $\B(\ss)$ the family of all the measurable functions on $\ss$ into $\rr$.

\subsection{Assumptions and Preliminaries}\label{sect-preliminaries}

 Similar to Definition \ref{defn-mg-A}, for a given $k \in \ss$, we can also define
the martingale solution for the L\'{e}vy type operator ${\LL}_{k}$ of \eqref{L}
as follows. For a given $x\in \rr^{d}$, we say a probability measure
${\bP}_{k}^{(x)}$ on $D([0,\infty),\rr^{d})$ is a solution to the
martingale problem for the operator ${\LL}_{k}$ starting from $x$,
if $\bP_{k}^{(x)}(X(0)=x)=1$ and for each function $f \in
C^{\infty}_{c}(\rr^{d})$,
\begin{equation}\label{martingale2}
M_{t}^{(k)(f)}:=f(X(t))-f(X(0))-\int_{0}^{t} {\LL}_{k} f(X(s))\d s \end{equation} is a $\{\G_{t}\}$-martingale with respect to ${\bP}_{k}^{(x)}$.

For the existence and uniqueness
of martingale solution corresponding to the
weakly coupled L\'{e}vy type operator $\A$ defined in (\ref{A}), we
make the following assumption.

  \begin{Assumption} \label{assump-Lhasonlyone} Suppose the following conditions hold:
\begin{itemize}
  \item[(i)] For each $k \in \ss$ and $x\in \R^{d}$, the L\'{e}vy type operator ${\LL}_{k}$ defined in \eqref{L}
  has a unique martingale solution ${\P}_{k}^{(x)}$ starting from $x$;
  \item[(ii)]  For each $k \in \ss$, the function $q_{kk}(x) \le 0$ is bounded from below; and
  \item[(iii)] \begin{equation}
\label{eq-nu-cond-new}
\sup_{(x,k)\in \R^{d}\times \ss} \int_{\R_{0}^{d}} (1\wedge |y|^{2} ) \nu (x,k, \d y) < \infty.
\end{equation}
\end{itemize} \end{Assumption}
\begin{Remark} The martingale problem for the operator $\LL_{k}$ of \eqref{L} has been well-studied in the literature. For example,
  \cite{Koma-73} and \cite{Stroock-75} contain   explicit
sufficient conditions for the existence and uniqueness of martingale solutions for  $\LL_{k}$.
\end{Remark}

We will prove in Section \ref{sect-general} that there exists a unique martingale solution for the operator $\A$ defined in
\eqref{A}. Throughout the rest of this paper,
as standing hypotheses, we assume that Assumption \ref{assump-Lhasonlyone} holds.


We finish the section with the following theorem, which will be needed in the proof of Theorem \ref{thm-general}, but also interesting in its own right. Let us introduce a counting measure as follows.
 For $t \ge 0$ and  $\Gamma \in \B(\R^{d}_{0})$ with $0 \notin \bar{\Gamma}$, we let \begin{equation}
\label{eq-eta(t,Gamma)}
\eta(t, \Gamma) : = \sum_{ s \le t} \one_{\Gamma}(\Delta X(s))=  \sum_{ s \le t} \one_{\Gamma}(X(s)- X(s-));
\end{equation} it counts the number of jumps for the $ X$ component such that  $\Delta X(s)\in \Gamma$, $0\le s \le t$.
\begin{Theorem}\label{thm-martingales}
Suppose $\P$ is a solution to the  martingale problem associated with $\A$ starting from $(x,k) \in \R^{d}\times \ss$, then
the following assertions are true: 
\begin{enumerate}
  \item[{\em (a)}] For each $f \in C^{2}_{b}(\R^{d}\times \ss)$ such that $f$ is uniformly positive,  $$f(X(t),\La(t) ) \exp\biggl\{ -\int_{0}^{t} \frac{\A f(X(u), \La(u))}{f(X(u), \La(u))}   \d u\biggr \}$$    is a $\P$-martingale.
  \item[{\em (b)}] For each $\theta \in \R^{d}$,
  \begin{align*}
\exp\biggl \{\mathtt{ i} \Big \langle \theta, X(t) - X(0) - \int_{0}^{t} b(X(u),\La(u)) \d u \Big\rangle + \frac{1}{2} \int_{0}^{t} \langle \theta, a(X(u),\La(u)) \theta \rangle \d u \\  - \int_{0}^{t} \int_{\R^{d}_{0}} [e^{\mathtt{i} \langle \theta, y \rangle} - 1 - \mathtt{i} \langle \theta, y \rangle  {\mathbf{1}}_{B(0, \e_{0})}(y) ] \nu (X(u), \La(u), \d y) \d u \biggr\}
\end{align*}    is a $\P$-martingale, where $\im:= \sqrt{-1}$.
  \item[{\em (c)}] Let $g$ be  a bounded measurable function on $\R^{d}$ which vanishes in a neighborhood of the origin. Then for any $\theta\in \R^{d}$,
  \begin{align*}
  \exp& \biggl \{\mathtt{ i} \Big \langle \theta, X(t) - X(0) - \int_{0}^{t} b(X(u),\La(u)) \d u \Big\rangle    + \frac{1}{2} \int_{0}^{t} \langle \theta, a(X(u),\La(u)) \theta \rangle \d u \\ &  \   + \int_{\R^{d}_{0}} g(y) \eta(t, \d y)   - \int_{0}^{t} \int_{\R^{d}_{0}} [e^{\mathtt{i} \langle \theta, y \rangle + g(y)} - 1 - \mathtt{i} \langle \theta, y \rangle  {\mathbf{1}}_{B(0, \e_{0})}(y) ] \nu (X(u), \La(u), \d y) \d u \biggr\}
\end{align*}
is a $\P$-martingale.
\item[{\em (d)}] Define $\wdt\eta(t,\Gamma) : = \eta(t, \Gamma) - \int_{0}^{t} \nu(X(u),\La(u), \Gamma) \d u$. Then for each $\theta \in \R^{d} $ and any measurable function $g$ on $\R^{d}_{0}$ satisfying the condition $|g(y)|^{2} \le C (1\wedge |y|^{2}) $ for some  positive constant  $C $,
\begin{align*}
   \exp &\biggl \{\mathtt{ i} \Big \langle \theta, X(t) - X(0) - \int_{0}^{t} b(X(u),\La(u)) \d u \Big\rangle  \\ &   \quad   + \frac{1}{2} \int_{0}^{t} \langle \theta, a(X(u),\La(u)) \theta \rangle \d u    + \int_{\R^{d}_{0}}  g(y) \wdt\eta(t, \d y)  \\ & \quad -\int_{0}^{t} \int_{\R^{d}_{0}} [e^{\mathtt{i} \langle \theta, y \rangle + g(y)} - 1 - \mathtt{i} \langle \theta, y \rangle  {\mathbf{1}}_{B(0, \e_{0})}(y) - g(y)] \nu (X(u), \La(u), \d y) \d u \!\biggr\}
\end{align*} is a $\P$-martingale. In particular, if $0 \notin \bar \Gamma$, then $\wdt\eta(t,\Gamma)$ is a $\P$-martingale.
\end{enumerate}
\end{Theorem}

\begin{proof} This theorem can be established using very similar arguments as those in the proof of Theorem 4.2.1 in \cite{Stroock-V}.
For brevity, we shall omit the  details here.
\end{proof} 

\section{Martingale Solution: Special Case}\label{sect-special}
 We first consider a special $Q$-matrix
  $\wdh{Q}=\bigl(\wdh{q}_{kl}\bigr)$,  in which
$\wdh{q}_{kl}=1$ for all $k, l \in \ss$ with $k \ne l$ and
$\wdh{q}_{kk}=-(n_{0}-1)$ for all $k\in \ss$. In other words,  we have
\begin{equation}\label{eq-Q-hat}
\wdh{Q}=\bigl(\wdh{q}_{kl}\bigr)=\left(\begin{array}{cccc}
{-(n_{0}-1)} & 1 & \cdots & 1\\
{  1} & {-(n_{0}-1)} & \cdots & 1\\
\vdots & \vdots & \ddots & \vdots\\
1 & 1 & \cdots & -(n_{0}-1)
\end{array} \right).
\end{equation}  
Corresponding to this matrix $\wdh Q$,  we
introduce an operator $\wdh{Q}$ on $\B(\ss)$ as follows:
\begin{equation}
\label{eq-Q-hat-op-defn}
\wdh{Q}f(k)=\sum_{l \in \ss} \wdh{q}_{kl}\bigl(f(l)-f(k)\bigr), \quad k \in \ss.
\end{equation}

For a given $k\in \ss$,  a probability measure ${\bQ}^{(k)}$
on $D([0,\infty),\ss)$ is said to be a solution to the martingale problem for
the operator $\wdh{Q}$ starting from $k$, if
${\bQ}^{(k)}(\La(0))=k)=1$ and for each function $f \in \B(\ss)$,
\begin{equation}\label{martingale3} N_{t}^{(f)}:=f(\La(t))-f(\La(0))-\int_{0}^{t} \wdh{Q}
f(\La(s))\d s \end{equation} is an $\{\N_{t}\}$-martingale with respect to ${\bQ}^{(k)}$. Here $\La$ is the coordinate process $\La(t,\omega) : = \omega(t)$ with $\omega \in D([0,\infty), \ss)$ and $t\ge 0$.

We have the following lemma from
\cite{ZhengZ-86}:

\begin{Lemma} \label{lemma-Q}
For any given $k\in \ss$, there exists a unique martingale solution ${\bQ}^{(k)}$ on $D([0,\infty),\ss)$ for the operator $\wdh{Q}$ starting from $k$. \end{Lemma}

Let  $\La$ be the coordinate process on $D([0,\infty), \ss)$ and let $\{\tau_{n}\}$ be the sequence of stopping times defined by
\begin{equation}\label{tau} \tau_{0} \equiv 0, \quad \text{ and for } n \ge 1, \quad \tau_{n}: =\inf \{ t >\tau_{n-1}:
\La(t) \neq \La(\tau_{n-1})\}. \end{equation} Then it is obvious that
 for any  $k\in \ss$, ${\bQ}^{(k)}\set{\lim_{n \to \infty}
\tau_{n}=+\infty}=1$. Moreover, we have
${\bQ}^{(k)}\bigl(\tau_{1}\ge t\bigr)=\exp(-(n_{0}-1)t)$ for all
$t\ge 0$ and
$${\bQ}^{(k)}\bigl(\La(\tau_{1})=l\bigr)={1}/{(n_{0}-1)} \text{  for  each }
 l\in \ss\setminus \{k\}.$$ Clearly, the distributions of $\tau_{1}$
and $\La(\tau_{1})$ under ${\bQ}^{(k)}$ are   regular.

Now we introduce an operator $\wdh \A$ on $C_{c}^{2}(\R^{d}\times \ss)$ as follows:
\begin{equation}
\label{eq-A-hat}
\wdh \A f(x,k ) : = \LL_{k} f(x,k) + \wdh Q f(x,k),
\end{equation} where the operators $\LL_{k}$ and $\wdh Q$ are defined in \eqref{L} and \eqref{eq-Q-hat-op-defn}, respectively.  Note that $\wdh \A$ of \eqref{eq-A-hat} is really a special case of the operator $\A$ defined in \eqref{A}. 
  We can define the martingale solution for the operator $\wdh \A$ similarly as in Definition \ref{defn-mg-A}.
For convenience of later presentation, let us also denote
 \begin{equation} \label{eq-Mt-hat}
   \wdh M_{t}^{(f)} : = f(X(t),\La(t)) - f(X(0), \La(0)) - \int_{0}^{t} \wdh \A f(X(s), \La(s)) \d s,
 \end{equation} where $f\in C_{c}^{\infty} (\R^{d} \times \ss)$ and $(X, \La)$ is the coordinate process on $D([0,\infty), \R^{d}\times \ss)$.

We will show that for each $(x,k) \in \rr^{d}
\times \ss$, there exists a unique martingale solution $\wdh \P^{(x,k)}$ for the operator $\wdh \A$ starting from $(x.k)$. Our construction of the desired probability measure $\wdh \P^{(x,k)}$ on $D([0,\infty),\rr^{d} \times \ss)$  as well as the proof  of uniqueness for such a solution relies heavily on
the martingale solutions $\{\bP_{k}^{(x)}: k\in\ss,
x\in\rr^{d}\}$ and $\{{\bQ}^{(k)}: k\in\ss\}$, and the stopping
times $\{\tau_{n}\}$ defined in \eqref{tau}.

 But first let us introduce a random point process and a  family of counting measures on $\ss$ as follows.
For $t>0$, $k\in \ss $,  and $A \subset
\ss$, set
\begin{equation}
\label{eq-n-defn}
 n (t,A):=\sum_{s \le t} {\mathbf{1}}_{\{\La(s) \in A, \La(s)\neq
\La(s-)\}},   
\end{equation}   
and
$$\nu(k;A):= \sum_{l \in A \setminus \{k\}} \wdh q_{kl}= \#\{A \backslash\{ k\} \} .$$ In view of Lemma
2.4 of \cite{ShigaT-85}, we know that $\int_{0}^{t} \nu(\La(s);A)\d s$ is the
compensator of the point process $n (t,A)$; namely,
\begin{equation}
\label{eq-tilde-mu}
 \mu (t,A):= n (t,A)-\int_{0}^{t} \nu(\La(s);A)\d s
\end{equation} is a martingale
measure with respect to $\Q^{(k)}$. Moreover, notice that the operator $\wdh Q$ defined in \eqref{eq-Q-hat-op-defn} can
be represented as
\begin{equation}
\label{eq-Q-hat-operator}
 \wdh Qf(k)=\sum_{l \in \ss}\wdh q_{kl}\bigl(f(l) - f(k)\bigr) =\int_{\mathbb S}\bigl(f(l)-f(k)\bigr)\nu(k;\d l).
\end{equation}

Now we present the main result of this section:
\begin{Theorem} \label{special}
For any given $(x,k)\in \rr^{d}
\times \ss$, there exists a unique martingale solution $\wdh\bP^{(x,k)}$ on $D([0,\infty),\rr^{d}
\times \ss)$ for the operator $\wdh\A$ starting from $(x,k)$.
\end{Theorem}

\begin{proof} The proof is divided into two steps. The first step establishes the existence of a martingale solution $\wdh \P$ for the operator $\wdh \A$ starting from $(x,k)$ while the second step deals with the uniqueness.

{\em Step 1.} For any given $(x,k)\in \rr^{d} \times \ss$, we define a series of probability measures on $(\Omega,\F)$ as follows:
\begin{equation}\label{Pn}
\bP^{(1)}=\bP_{k}^{(x)}\times \bQ^{(k)}, \quad
 \text{ and for }n\ge 1, \quad
\bP^{(n+1)}=\bP^{(n)}\otimes{}_{\tau_{n}}
\bigl(\bP_{\La(\tau_{n})}^{(X(\tau_{n}))}\times
\bQ^{(\La(\tau_{n}))} \bigr),    \end{equation} where
$\Omega=D([0,\infty),\rr^{d} \times \ss)$.
Thanks to Theorem 6.1.2 of \cite{Stroock-V}, $\bP^{(n+1)} =\bP^{(n)} $ on $\F_{\tau_{n}}$.

Let $f\in C_{c}^{2} (\R^{d} \times \ss)$. We have
\begin{displaymath}
f(X(\tau_{1} \wedge t), k) - f(X(0),k) - \int_{0}^{\tau_{1} \wedge t} \LL_{k} f(X(s),k) \d s
\end{displaymath}
is a martingale with respect to $\P^{(x)}_{k}$ and hence  $\P^{(1)}$.
On the other hand, using \eqref{eq-Q-hat-operator}, we can write
\begin{align*}
&  \int_{0}^{\tau_{1} \wedge t}   \wdh Q f(X(s), \La(s))\d s  \\    & \ \  = \int_{0}^{\tau_{1}\wedge t}  \int_{\ss} [f(X(s), l) - f(X(s),\La(s))] \nu(\La(s), \d l) \d s \\
     & \ \  = -\int_{0}^{\tau_{1}\wedge t}  \int_{\ss} [f(X(s), l) - f(X(s),\La(s))] \big(n (\d s , \d l) - \nu(\La(s), \d l) \d s\big) \\
    & \ \  \ \ + \int_{0}^{\tau_{1}\wedge t}  \int_{\ss} [f(X(s), l) - f(X(s),\La(s))] n (\d s , \d l)\\
    & \ \ =  -\int_{0}^{\tau_{1}\wedge t}  \int_{\ss} [f(X(s), l) - f(X(s),\La(s))] \mu (\d s , \d l)    \\
    & \ \ \ \
   + f(X(\tau_{1} \wedge t), \La (\tau_{1}\wedge t)) - f(X(\tau_{1}\wedge t), \La(\tau_{1}\wedge t-)).
\end{align*}
Then using the definitions of the operators $\wdh\A$, $\LL_{k}$ and $\wdh Q$,  we have
\begin{align*}
\wdh M_{\tau_{1}\wedge t}^{(f)} & = f (X(\tau_{1}\wedge t), \La(\tau_{1}\wedge t)) -  f(X(0),\La(0)) - \int_{0}^{\tau_{1} \wedge t}  {\wdh\A} f(X(s),\La(s)) \d s \\
    & = f(X(\tau_{1} \wedge t), \La (0)) - f(X(0),\La(0)) - \int_{0}^{\tau_{1} \wedge t} \LL_{\La(0)} f(X(s),\La(0)) \d s \\
     & \quad +  f(X(\tau_{1}\wedge t), \La(\tau_{1}\wedge t))  -  f(X(\tau_{1} \wedge t), \La (0))  \\
     & \quad + \int_{0}^{\tau_{1} \wedge t} \LL_{\La(0)} f(X(s),\La(0)) \d s - \int_{0}^{\tau_{1} \wedge t}  {\wdh\A} f(X(s),\La(s)) \d s  \\
    &  = f(X(\tau_{1} \wedge t), \La (0)) - f(X(0),\La(0)) - \int_{0}^{\tau_{1} \wedge t} \LL_{\La(0)} f(X(s),\La(0)) \d s \\
     & \quad +     f(X(\tau_{1}\wedge t), \La(\tau_{1}\wedge t))  -  f(X(\tau_{1} \wedge t), \La (0))  - \int_{0}^{\tau_{1} \wedge t} \wdh Q f(X(s), \La(s))\d s\\
     & =  f(X(\tau_{1} \wedge t), \La (0)) - f(X(0),\La(0)) - \int_{0}^{\tau_{1} \wedge t} \LL_{\La(0)} f(X(s),\La(0)) \d s \\
    &   \ \ \ +  \int_{0}^{\tau_{1}\wedge t}  \int_{\ss} [f(X(s), l) - f(X(s),\La(s))]  \mu (\d s , \d l).
\end{align*}
 Recall  that $\mu$ is a martingale measure with respect to $\Q^{(k)}$ and hence $\P^{(1)}$. Thus it follows that $\wdh  M_{\tau_{1}\wedge \cdot}^{(f)} $
is a martingale with respect to $ \P^{(1)}$.

Next,
\begin{displaymath}
f(X(\tau_{2} \wedge t), \La (\tau_{1})) - f(X(\tau_{1}),\La(\tau_{1})) - \int_{\tau_{1} }^{\tau_{2} \wedge t} L_{\La(\tau_{1})} f(X(s),\La(\tau_{1})) \d s, \quad t\ge \tau_{1}
\end{displaymath}  is a martingale with respect to $\P_{\La(\tau_{1})}^{(X(\tau_{1}))}\times
\bQ^{(\La(\tau_{1}))}$. Then a similar argument as above gives that
$$ f(X(\tau_{2}\wedge t), \La(\tau_{2}\wedge t)) - f(X(\tau_{1}), \La(\tau_{1})) - \int_{\tau_{1}}^{\tau_{2}\wedge t} {\wdh\A} f (X(s),\La(s)) \d s, \quad t \ge \tau_{1}$$ is a martingale with respect to  $\P_{\La(\tau_{1})}^{(X(\tau_{1}))}\times
\bQ^{(\La(\tau_{1}))}$.  Notice that the above displayed equation is equal to $ \wdh  M^{(f)}_{\tau_{2}\wedge t} -\wdh   M_{\tau_{1}\wedge t}^{(f)}.$ Then in view of Theorem 6.1.2 of \cite{Stroock-V}, $\wdh  M^{(f)}_{\tau_{2}\wedge \cdot}$ is a martingale with respect to $\P^{(2)}$.
In a similar fashion, we can show that $\wdh  M^{(f)}_{\tau_{n}\wedge \cdot}$ is a martingale with respect to $\P^{(n)}$ for any $n \ge 1$. 

Next we show that  $\lim_{n\to \infty} \P^{(n)} \{ \tau_{n } \le t \} =0$ for any $t\ge 0$. To this end, we consider  functions of the form $f(x,k) = g(k)$, where $g \in \B(\ss)$. Then $M^{(f)}_{\tau_{n}\wedge \cdot}$ is a $\P^{(n)}$ martingale. But for any $t \ge 0$, $$\wdh M^{(f)}_{t} = N^{(g)}_{t} = g(\La (t) ) - g(\La(0)) - \int_{0}^{t} \wdh Q g(\La(s)) \d s$$  is  a martingale with respect to $\Q^{(k)}$. In particular, $ N^{(g)}_{\tau_{n}\wedge \cdot}$  is  a martingale with respect to $\Q^{(k)}$  as well.  On the other hand, for any $A \in \N$, we define
$\wdh \Q (A) : = \P^{(n)}\{ D([0,\infty), \R^{d}) \times A \}$.  Then $N^{(g)}_{\tau_{n}\wedge \cdot}$ is a martingale with respect to $\wdh \Q$. By   the uniqueness result for the martingale problem for $\wdh Q$ in Lemma \ref{lemma-Q}, we have $\wdh \Q = \Q^{(k)}$.
Therefore it follows that \begin{displaymath}
\P^{(n)} \{ \tau_{n } \le t \} = \wdh \Q\{ \tau_{n } \le t   \} =  \Q^{(k)}\{ \tau_{n } \le t   \}  \to 0, \text{ as  } n \to \infty.
\end{displaymath}

 Recall that the probabilities $\P^{(n)}$ constructed in \eqref{Pn} satisfies $\P^{(n+1)} = \P^{(n)}$ on $\F_{\tau_{n}}$.
Hence by Tulcea's extension theorem (see, e.g., \cite[Theorem 1.3.5]{Stroock-V}),  there exists a
unique $\wdh \bP$ on $(\Omega,\F)$ such that $\wdh \bP$ equals $\bP^{(n)}$ on
${\F}_{\tau_{n}}$.
Thus it follows that $\wdh M^{(f)}_{\tau_{n}\wedge \cdot}$ is a martingale with respect to $\wdh \P$ for every $n \ge 1$. In addition,   for any $t \ge 0$, we have \begin{equation}
\label{eq-tau-n-to-infty}
 \wdh  \P \{\tau_{n}  \le t \} = \P^{(n)} \{\tau_{n}  \le t \}  =0.
\end{equation} Thus $\tau_{n} \to \infty$ a.s. $\wdh\P$ and hence $\wdh M^{(f)}_{\cdot}$ is a martingale with respect to $\wdh \P$. This establishes that $\wdh \P$
is the desired martingale solution staring from $(x,k)$ to the   martingale problem for $\wdh \A$.  When we
wish to emphasize the initial data dependence $X(0)=x$ and
$\La(0)=k$, we write this martingale solution as $\wdh \bP^{(x,k)}$.

{\em Step 2.} Next we show that there is at most one solution to the martingale problem associated with $\wdh \A$ starting from $(x,k)$. To this purpose, we let $  \wdt \P^{(x,k)}\in \mathcal P(\Omega, \F)$ be another  solution to the martingale problem associated with $\wdh \A$ starting from $(x,k)$. We show that $\wdh \P^{(x,k)}$ and $  \wdt \P^{(x,k)} $ agree on $\F_{\tau_{1}}$.
  Recall that $\wdh \P^{(x,k)}$ agrees with $\P^{(1)}= \P^{(x)}_{k} \times \Q^{(k)}$ on $\F_{\tau_{1}}$ and that $\P^{(x)}_{k}\in \mathcal P(D([0,\infty); \R^{d}))$ is the unique solution to the martingale problem associated with $\LL_{k}$ starting from $x$.  Also notice that  any $A\in \F_{\tau_{1}}$ is necessarily of the form $A_{1} \times \delta_{k} $, where $A_{1} \subset D([0, \infty), \R^{d})$ and $\delta_{k} $  contains all functions $\omega$  in $  D([0,\infty), \ss) $ satisfying $\omega(t) = k$ for all $0\le t < \tau_{1}$ and $\omega(\tau_{1}) \in \ss \backslash \{k\}$. Since $\Q^{(k)} (\delta_{k}) =1$,    it follows that
\begin{align}\label{eq1-uniqueness}
\wdh \P^{(x,k)} (A) =   \P^{(x)}_{k} \times \Q^{(k)}  (A_{1} \times \delta_{k})  =    \P^{(x)}_{k} (A_{1}).
\end{align} On the other hand, since $\wdt \P^{(x,k)}$ is a solution to the martingale problem associated with $\wdh \A$ starting from $(x,k)$, for any $g\in C^{2}_{c}(\R^{d})$, $\wdh M^{(g)}_{t}$ is a $\wdt \P^{(x,k)}$ martingale. In particular,
\begin{align*}
\wdh M^{(g)}_{\tau_{1}\wedge t}  & = g(X(t\wedge \tau_{1})) - g(X(0)) - \int_{0}^{\tau_{1}\wedge t}{\wdh\A} g(X(s)) \d s \\ &=  g(X(t\wedge \tau_{1})) - g(X(0)) - \int_{0}^{\tau_{1}\wedge t} \LL_{k} g(X(s)) \d s
\end{align*}  is a  $\wdt \P^{(x,k)}$ martingale.  Now for any $A_{1} \subset D([0, \infty), \R^{d})$ with $A_{1} \in \G$, we define \begin{equation}
\label{eq2-uniqueness}
 \widetilde \P (A_{1}) : = \wdt \P^{(x,k)} (A_{1}\times \delta_{k}).
\end{equation} Then $\wdh  M^{(g)}_{\tau_{1}\wedge \cdot} $ is also a $\widetilde \P$ martingale and  hence $\widetilde\P$ is  a solutions to the martingale problem associated with $\LL_{k} $ starting from $x$ up to  $\tau_{1}$.  Now by the uniqueness of the martingale solution to $\LL_{k}$ starting from $x$, we conclude from \eqref{eq1-uniqueness} and \eqref{eq2-uniqueness}
 that $\wdh \P^{(x,k)} (A)  =\wdt \P^{(x,k)} (A) $ for any $A \in \F_{\tau_{1}}$. This shows that the martingale solution to $\wdh\A$ starting from $(x,k)$ is uniquely determined on  $\F_{\tau_{1}}$.

Now suppose that  the martingale solution $\wdh\P^{(x,k)}$ to $\wdh\A$ starting from $(x,k)$ is uniquely determined on  $\F_{\tau_{n}}$. By virtue of Theorem 6.2.1 of \cite{Stroock-V} (also Lemma 5.4.19 of \cite{Karatzas-S}), there is a $\wdh \P^{(x,k)}$-null set $N \in \F_{\tau_{n}}$ such that $$\wdh \P^{(X(\tau_{n}(\omega)), \La(\tau_{n}(\omega)))}: =\delta_{(X(\tau_{n}(\omega)), \La(\tau_{n}(\omega)), \omega)}\otimes_{\tau_{n}(\omega)} \wdh  \P_{\omega}$$ solves the martingale problem for $\wdh \A$ starting from $(X(\tau_{n}(\omega)), \La(\tau_{n}(\omega)))$ whenever $\omega \notin N$, where $\wdh \P_{\omega}$ is the regular conditional probability  distribution of $\wdh \P^{(x,k)}$ given $\F_{\tau_{n}}$, whose existence follows from \cite[Theorem 5.3.18]{Karatzas-S}. By the argument in the previous paragraph, $\wdh  \P^{(X(\tau_{n}(\omega)), \La(\tau_{n}(\omega)))}$ is uniquely determined on $\F_{\tau_{n+1}}$.  Note that by virtue of Theorem 6.1.2 of \cite{Stroock-V},
\begin{displaymath}
\wdh \P^{(x,k)} =\wdh  \P^{(x,k)}\otimes_{\tau_{n}(\cdot)} \wdh \P^{(X(\tau_{n}(\cdot)), \La(\tau_{n}(\cdot)))},
\end{displaymath} In other words, the right-hand side of the above displayed equation satisfies 
\begin{itemize}
  \item[(i)] $\wdh \P^{(x,k)}\otimes_{\tau_{n}(\cdot)}\wdh  \P^{(X(\tau_{n}(\cdot)), \La(\tau_{n}(\cdot)))} (A) = \wdh \P^{(x,k)} (A)$,  for any $A\in \F_{\tau_{n}}$, and
  \item[(ii)]  $\delta_{(X(\tau_{n}(\omega)), \La(\tau_{n}(\omega)), \omega)}\otimes_{\tau_{n}(\omega)}\wdh  \P_{\omega}$ is a regular conditional probability distribution of $ \wdh \P^{(x,k)}\otimes_{\tau_{n}(\cdot)} \wdh \P^{(X(\tau_{n}(\cdot)), \La(\tau_{n}(\cdot)))}$ given $\F_{\tau_{n}}$.
\end{itemize} Thus by the induction hypothesis, we conclude that $\wdh \P^{(x,k)}$
 is uniquely determined on $\F_{\tau_{n+1}}$.

 Now we define for any $n \in \mathbb N$ and $A \in \F_{\tau_{n}}$ that $\P_{n}(A) : = \wdh \P^{(x,k)}(A)$. Apparently $\P_{n}$ satisfies that  $\P_{n}=  \P_{ n+1}$ on $\F_{\tau_{n}}$ and that for any $t\ge 0$,  $\P_{n} \{ \tau_{n }\le t\} =\wdh  \P^{(x,k)} \{ \tau_{n} \le t\} \to 0$ as $n \to \infty$, where we used \eqref{eq-tau-n-to-infty}.  Therefore
 by Tulcea's extension theorem (e.g., \cite[Theorem 1.3.5]{Stroock-V}), the sequence $\P_{n}$ has a unique extension $\wdh \P$ on $(\Omega, \F)$ such that $\wdh \P = \P_{n}$ on $\F_{\tau_{n}}$. The measure $\wdh \P$ solves the martingale problem for the operator $\wdh \A$ starting from $(x,k)$. This completes the proof.
\end{proof}

\section{Martingale Solution: General Case}\label{sect-general}

In this section we construct the martingale solution for the general case.
To proceed, for any given $t\ge 0$, we define a function $M_{t}$ on the sample path space as follows:
\begin{equation}\label{Mt}
\barray
M_{t}\bigl(X(\cdot),\La (\cdot) \bigr)\ad :=
\prod_{i=0}^{n(t)-1} q_{\La (\tau_{i})\La (\tau_{i+1})} \bigl(X(\tau_{i+1})\bigr)\\
\aad \quad \times \displaystyle \exp \biggl(-\sum_{i=0}^{n(t)}
\int_{\tau_{i}}^{\tau_{i+1}\wedge t} \bigl[q_{\La (\tau_{i})}
(X(s))-n_{0}+1 \bigr] \d s \biggr),
\earray
\end{equation}
where
$$q_{k}(x)= \sum_{l \in \ss\setminus \{k\}} q_{kl}(x),\quad n(t)= \max \{i \in \mathbb N: \tau_{i} \le t\},$$
and $\{\tau_{i}\}$ is the sequence of stopping times defined in  \eqref{tau}. In case $n(t) = 0$, we use the convention that $\prod_{i=0}^{-1} a_{i} : = 1$ in \eqref{Mt}.

\begin{Lemma} \label{Mtmartingale}
We have that $\bigl(M_{t}, \F_{t}, \wdh{\bP}\bigr)$ is a non-negative martingale with mean one.
\end{Lemma}

\begin{proof} {\em Step 1.} We first observe that if $q_{kl}(x) > 0$ for all $k\not = l$ and $x\in \R^{d}$, then \begin{align*}
 \prod_{i=0}^{n(t)-1} q_{\La (\tau_{i})\La (\tau_{i+1})} \bigl(X(\tau_{i+1})\bigr) & = \exp \biggl\{ \sum_{i=0}^{n(t) -1}  \log q_{\La (\tau_{i})\La (\tau_{i+1})} \bigl(X(\tau_{i+1})\bigr)\biggr\} \\
    &  =  \exp \biggl\{\int_{[0,t] \times \ss} \log q_{\La(s-)l}\bigl(X(s)\bigr)n(\d s,\d l) \biggr\},
\end{align*}
 where  $n(t,A)$ is the  Poisson random measure  defined in \eqref {eq-n-defn}.
 Then it follows from the definition of $M$ in \eqref{Mt} that \begin{equation}\label{eq-Mt-exponential}
M_{t}\bigl(X(\cdot),\La(\cdot)\bigr) = \exp\{Z(t)\},  \end{equation} where $$ Z(t) : =
  \int_{[0,t] \times \ss} \log q_{\La(s-)l}\bigl(X(s)\bigr)n(\d s,\d l)
-\int_{0}^{t} \bigl[q_{\La(s)}
(X(s))-n_{0}+1\bigr]\d s, $$
Now we apply  It\^{o}'s formula for jump processes (see, e.g., \cite[Theorem II.5.1]{IkedaW-89}) to the process $M_{t}$:
\begin{align}
\nonumber \lefteqn{M_{t}\bigl(X(\cdot),\La (\cdot) \bigr) -1 = e^{Z(t) } - e^{Z(0)}} \\
 \label{eq1-Mt}  &\ \  = \int_{0}^{t} \int_{\ss}  e^{Z(s-)} [q_{\La (s-)l} \bigl(X(s)\bigr)-1] n(\d s, \d l) -  \int_{0}^{t}  e^{Z(s)}  \bigl[q_{\La(s)} (X(s))-n_{0}+1\bigr]\d s.    \end{align}

  Recall from Section \ref{sect-special} that for any $s\ge 0$,  $\wdh \P\{ \La(s) = l, \La(s) \not= \La(s-) \} = \frac{1}{n_{0}-1}$.
Thus we have
\begin{align*}
\E^{\wdh \P} [n(t,A)] & = \E^{\wdh \P} \Biggl[\sum_{s \le t} {\mathbf{1}}_{\{\La(s) \in A, \La(s)\neq
\La(s-)\}}\Biggr]
  =  \E^{\wdh \P} \Biggl[\sum_{k\in \ss}\sum_{s \le t} {\mathbf{1}}_{\{\La(s) \in A, \La(s)\neq
\La(s-), \La(s-) = k\}}\Biggr] \\
& = (n_{0}-1)\int_{0}^{t} \int_{A}  \frac{1}{n_{0}-1}  \d l \d s = \int_{0}^{t}\int_{A}  \d l  \d s,
\end{align*} where $\d l$ is the counting measure on $\ss$.
  Then it follows that
  \begin{align*}
\int_{0}^{t}  e^{Z(s)}  \bigl[q_{\La(s)} (X(s))-n_{0}+1\bigr]\d s & =  \int_{0}^{t} e^{Z(s)} \sum_{l \not= \La(s-)} \left[ q_{\La(s-) l} (X(s)) - 1\right] \d s \\
 & = \int_{0}^{t}\int_{\ss} e^{Z(s)}\left[ q_{\La(s-) l} (X(s)) - 1\right] \d l \d s.
\end{align*}  Putting these observations into \eqref{eq1-Mt} and using \eqref{eq-Mt-exponential}, we obtain
\begin{align}
\label{eq-Mt-ito}   M_{t}\bigl(X(\cdot),\La (\cdot) \bigr) -1  = \int_{[0,t] \times \ss}
M_{s-}\bigl(X(\cdot),\La (\cdot) \bigr) \left[q_{\La (s-)l}
\bigl(X(s)\bigr)-1\right] \widetilde{n}(\d s, \d l),
\end{align} where $\widetilde{n}(t, A)=n(t, A)- \E^{\wdh \P} [n(t, A)]$ is the
compensated Poisson random measure with respect to $\wdh \P$ and also a martingale measure on
$[0, \infty) \times \ss$.

{\em Step 2.}  In general,  there may exist some $i\not= j$ and $x\in \R^{d}$ so that $q_{ij}(x) =0$. We define $q_{kl}^{\e}(x) : = q_{kl}(x) + \e$ for all $k,l \in \ss$ with $k\not=l$ and $x\in \R^{d}$. Also, we let $q_{kk}^{\e}(x)  :=q_{kk}(x) - (n_{0}-1) \e $ for all $k\in \ss $ and $x\in \R^{d}$.   Then  as $\e \downarrow 0$, we have
  \begin{displaymath}
q_{kl}^{\e}(x) \to q_{kl}(x), \text{ and } q_{kk}^{\e}(x) \to q_{kk}(x)
\end{displaymath} uniformly with respect to $x\in \R^{d}$
for all $l \not=k  \in\ss$.
Next we define \begin{displaymath}
M_{t}^{\e}(X(\cdot),\La(\cdot)) : = \exp \biggl\{\int_{[0,t]\times \ss} \log q^{\e}_{\La(s-)l} (X(s)) n(\d l, \d s)  - \int_{0}^{t} \left[ q_{\La(s)}^{\e} (X(s)) - n_{0}+ 1\right] \d s \biggr\}.
\end{displaymath}
Thanks to  Assumption \ref{assump-Lhasonlyone} and the bounded convergence theorem, we have    $M_{t}^{\e} (X(\cdot),\La(\cdot))  \to M_{t}(X(\cdot),\La(\cdot)) $ as $\e \downarrow 0$. Moreover, by \eqref{eq-Mt-ito}  in  Step 1, we have
 $$ M_{t}^{\e}(X(\cdot),\La(\cdot)) - 1 = \int_{0}^{t}\int_{\ss}  M_{s-}^{\e}(X(\cdot),\La(\cdot))  \left[q_{\La (s-)l}^{\e}
\bigl(X(s)\bigr)-1\right] \widetilde{n}(\d s, \d l).$$
Now passing to the limit as $\e \downarrow 0$, we obtain from the bounded convergence theorem that
\begin{equation}\label{eq2-Mt-ito}
 M_{t}\bigl(X(\cdot),\La (\cdot) \bigr) -1  = \int_{[0,t] \times \ss}
M_{s-}\bigl(X(\cdot),\La (\cdot) \bigr) \left[q_{\La (s-)l}
\bigl(X(s)\bigr)-1\right] \widetilde{n}(\d s, \d l).
\end{equation}

{\em Step 3.} From \eqref{eq2-Mt-ito}, we can see that
$M_{t}(X(\cdot),\La (\cdot) )$ is a martingale with mean $1$ under $
\wdh \P$. This completes the proof.  \end{proof}

\begin{Lemma}\label{lem-Mt-integrable}
For any $T > 0$ and $(x,k)\in \R^{d} \times \ss$, the function $M_{T}(X(\cdot),\La(\cdot))$ defined in  \eqref{Mt} is integrable with respect to the measure $\wdh \P$.
\end{Lemma}
\begin{proof}
The proof is similar to that of \cite[Lemma 4.4]{Xi-09} and we shall omit the details here.
\end{proof}

Let $\e>0$ and notice that in view of \eqref{eq-Levy-measure},
$\nu(x,k, \R^{d} \backslash B(0, \e) )< \infty$ for each $(x,k)\in \R^{d}\times \ss$.
Then we can define a sequence of stopping times as follows. Let $\zeta_{0}^{(\e)}:=0 $ and for $n \ge 0$,\begin{equation}
\label{eq-jump-times}
 \zeta_{n+1}^{(\e)}: = \inf\{t \ge \zeta_{n}^{(\e)}: | \Delta X(t) | = | X(t) - X(t-)| \ge \e \}.
\end{equation}

\begin{Lemma}\label{lem-big-jump-distribution}
Let 
$X_{\e}(t)  := X(t) - \int_{|y| \ge \e} y \eta(t, \d y)$ for $t \ge 0$  and define $\F_{\zeta_{1}^{(\e)}-} : = \sigma\{X_\e(t\wedge \zeta_{1}^{(\e)}), \La(t\wedge \zeta_{1}^{(\e)}): t \ge 0 \}$. Then we  have
\begin{align}
\label{1st-switch}
  & \wdh\P\{\tau_{1} > t\} = \exp\{-(n_{0}-1) t \},   \\
\label{1st-big-jump}    &   \wdh \P\Big\{ \zeta_{1}^{(\e)} > t| \F_{\zeta_{1}^{(\e)}-}\Big\} =\exp\biggl\{ -\int_{0}^{t} \nu(X_{\e}(u \wedge \zeta_{1}^{(\e)}), \La(u\wedge \zeta_{1}^{(\e)}), \R^{d}\backslash B(0, \e)) \d u\biggr\}.
\end{align}
\end{Lemma}

\begin{proof}
Equation \eqref{1st-switch} follows directly from the construction of $\wdh \P$ in Theorem \ref{special}.
Now we prove  \eqref{1st-big-jump}.
Let $\Gamma : = \R^{d } \backslash B(0 ,\e)$, and recall $ \eta(t, \Gamma)$ defined in \eqref{eq-eta(t,Gamma)}. Let us also denote  \begin{align*}
  \wdt \eta(t, \Gamma) : &= \eta(t,\Gamma) - \int_{0}^{t} \nu(X(u),\La(u),\Gamma) \d u.
       \end{align*}
      Thanks to Theorem \ref{thm-martingales},   $\wdt \eta(t, \Gamma)$ is a $\wdh\P$-martingale.  Consequently, for any $t \ge 0$, we have
       \begin{equation}\label{eq1-lem33-pf}\begin{aligned}
\E^{\wdh \P} \Big[\eta (t\wedge \zeta_{1}^{(\e)}; \Gamma) \big| \F_{\zeta_{1}^{(\e)}-}\Big] & = \E^{\wdh \P} \biggl[\int_{0}^{t \wedge \zeta_{1}^{(\e)}} \nu(X(u),\La(u), \Gamma) \d u | \F_{\zeta_{1}^{(\e)}-}\biggr] \\
& = \E^{\wdh \P}  \biggl[\int_{0}^{t} \one_{\{ \zeta_{1}^{(\e)} > u\}} \nu(X(u\wedge \zeta_{1}^{(\e)}),\La(u\wedge \zeta_{1}^{(\e)}), \Gamma) \d u | \F_{\zeta_{1}^{(\e)}-}\biggr] \\
& = \int_{0}^{t} \wdh\P \Big\{ \zeta_{1}^{(\e)} > u \big |\F_{\zeta_{1}^{(\e)}-}\Big  \} \nu(X_{\e}(u\wedge \zeta_{1}^{(\e)}), \La(u\wedge \zeta_{1}^{(\e)}), \Gamma) \d u.
\end{aligned}\end{equation}
On the other hand, note that  $$\eta (t\wedge \zeta_{1}^{(\e)}; \Gamma) =\begin{cases}
    1  & \text{ if } \zeta_{1}^{(\e)} \le t, \\
    0  & \text{ otherwise}.
\end{cases}$$ Thus we have
\begin{equation}\label{eq2-lem33-pf}\begin{aligned}\wdh\P\set{\zeta_{1}^{(\e)} > t  \big| \F_{\zeta_{1}^{(\e)}-}} & =\E^{\wdh\P}\Big[\one_{\{\zeta_{1}^{(\e)} > t  \}} \big| \F_{\zeta_{1}^{(\e)}-} \Big]  \\ & = \E^{\wdh \P}\Big[ \big(1- \eta (t\wedge \zeta_{1}^{(\e)}; \Gamma) \big)\big| \F_{\zeta_{1}^{(\e)}-} \Big] =1 - \E^{\wdh \P}\Big[  \eta (t\wedge \zeta_{1}^{(\e)}; \Gamma) \big| \F_{\zeta_{1}^{(\e)}-} \Big]. \end{aligned}\end{equation}
Combining \eqref{eq1-lem33-pf} and \eqref{eq2-lem33-pf}, we arrive at
\begin{equation}\label{eq3-lem33-pf}
1- \wdh\P\set{\zeta_{1}^{(\e)} > t  \big| \F_{\zeta_{1}^{(\e)}-}}  =  \int_{0}^{t} \wdh\P \Big\{ \zeta_{1}^{(\e)} > s \big |\F_{\zeta_{1}^{(\e)}-}\Big  \} \nu(X_{\e}(s\wedge \zeta_{1}^{(\e)}), \La(s\wedge \zeta_{1}^{(\e)}), \Gamma) \d s.
\end{equation}

Let us denote $u(t) : = \wdh\P\{\zeta_{1}^{(\e)} > t  | \F_{\zeta_{1}^{(\e)}-}\} $ and $v(t): = \nu(X_{\e}(t\wedge \zeta_{1}^{(\e)}), \La(t\wedge \zeta_{1}^{(\e)}), \Gamma)$. Then we can rewrite \eqref{eq3-lem33-pf} as $u(t) + \int_{0}^{t} u(s)v(s) \d s =1$, which, in turn, implies that
\begin{align*}
\frac{ \d}{\d t} \biggl( e^{\int_{0}^{t} v(r) \d r} \int_{0}^{t} u(r) v(r) \d r \biggr) = e^{\int_{0}^{t} v(r) \d r} v(t) \biggl[u(t) + \int_{0}^{t} u(r) v(r) \d r \biggr] = e^{\int_{0}^{t} v(r) \d r} v(t).
   \end{align*}
   Then it follows that $$e^{\int_{0}^{t} v(r) \d r} \int_{0}^{t} u(r) v(r) \d r = \int_{0}^{t }  e^{\int_{0}^{s} v(r) \d r} v(s) \d s = e^{\int_{0}^{t} v(r) \d r} - 1,$$ and hence
   \begin{displaymath}
u(t) = 1- \int_{0}^{t} u(s)v(s) \d s = 1- \Big(1 - e^{-\int_{0}^{t} v(r) \d r} \Big) = e^{-\int_{0}^{t} v(r) \d r}.
\end{displaymath} This establishes  \eqref{1st-big-jump} and hence completes the proof of the lemma.
\end{proof}

\begin{Lemma}\label{lem-no-common-jumps}
Let $\e>0$ and define the stopping times $\zeta_{n}^{(\e)}$ as in \eqref{eq-jump-times} and recall the sequence of stopping times $\{ \tau_{n}\}$ defined in \eqref{tau}. Then under $\wdh \P$, $\{\zeta_{n}^{(\e)}, n\ge 1 \}$ and $\{\tau_{n}: n \ge 1 \}$ are mutually disjoint with probability 1.
\end{Lemma}

\begin{proof}
It is enough to show that for any $T> 0$, $  \{\zeta_{n}^{(\e)}: n\ge 1, \zeta_{n}^{(\e)} \le T \}$ and $ \{\tau_{n}: n \ge 1, \tau_{n}\le T \}$ are mutually disjoint with probability 1. To this end, we let $M > 1$ and for $m=0,1,\dots, M-1$, we denote
\begin{align*}
& J^{(\e)}((m-1)/M, m/M] : = \max\set{n \in \mathbb N: \zeta_{n}^{(\e)} \le m/M} - \max\set{n \in \mathbb N: \zeta_{n}^{(\e)} \le (m-1)/M},\\
& S((m-1)/M, m/M] : = \max\set{n \in \mathbb N: \tau_{n}  \le m/M} - \max\set{n \in \mathbb N: \tau_{n}  \le (m-1)/M},
\end{align*}
and \begin{align*}
&  A_{m}: = \set{J^{(\e)}((m-1)/M, m/M] \ge 1  },
    & B_{m}: = \set{S((m-1)/M, m/M] \ge 1}.
\end{align*}  Thanks to \eqref{eq-nu-cond-new}, it follows that there exists  some positive constant $K_{\e}$ such that
\begin{equation}
\label{eq-nu-bdd}
\nu(x,k,  \R^{d} \backslash B(0, \e) )\le K_{\e}  < \infty, \text{ for all }(x,k) \in \R^{d}\times \ss.
\end{equation}
Then we have from \eqref{1st-switch}, \eqref{1st-big-jump},  \eqref{eq-nu-bdd}, and Lemma  \ref{lem-big-jump-distribution} that
\begin{align*}
 \lefteqn{ \wdh\P\set{ \{\zeta_{n}^{(\e)}: n\ge 1, \zeta_{n}^{(\e)} \le T \}  \cap  \{\tau_{n}: n \ge 1, \tau_{n}\le T \} \not= \emptyset} }\\
    &   \le \wdh \P \set{\text{there are one   jump and one   switch in the interval } \Big(\frac{ m-1}{M}, \frac{m}{M}\Big] \text{ for  some }m}\\
    &  \le \sum_{m=0}^{M-1} \wdh \P \set{A_{m} \cap B_{m}}
     = \sum_{m=0}^{M-1} \wdh \P (A_{m}) \wdh \P(B_{m}| A_{m})\\
    & \le \sum_{m=0}^{M-1} \biggl(1- \exp\biggl\{- \int_{\frac{m-1}{M}}^{\frac{m}{M}} \nu(X_{\e}(s), \La(s), \R^{d} \backslash B(0,\e)) \d s\biggr\} \biggr)\biggl(1- \exp\bigg\{-(n_{0}-1) \frac{1}{M}\biggr\} \biggr)\\
    & \le \sum_{m=0}^{M-1}   \biggl( 1- \exp\set{- K_{\e} \frac{1}{M}}\biggr)\biggl(1- \exp\bigg\{-(n_{0}-1) \frac{1}{M}\biggr\} \biggr).
\end{align*} Furthermore, using the elementary inequality $1- e^{-a}  \le a$ for $a \ge 0$, we obtain
\begin{align*}
  \wdh\P& \set{ \{\zeta_{n}^{(\e)}: n\ge 1, \zeta_{n}^{(\e)} \le T \}  \cap  \{\tau_{n}: n \ge 1, \tau_{n}\le T \} \not= \emptyset} 
  \le \sum_{m=0}^{M-1} \frac{n_{0}-1}{M}  \frac{K_{\e}}{M}  = \frac{(n_{0}-1)  K_{\e}}{M},
 \end{align*} which can be arbitrarily small since the   denominator $M$ is arbitrary. This implies the desired conclusion and hence completes the proof. \end{proof}

 Note that since $X \in D([0,\infty), \R^{d})$, the set of discontinuity points of $X$ is at most countable for almost all $\omega \in \Omega$, see, e.g. \cite{Rudin-69}. Therefore we can again define the sequence of jump times for $X$ as follows.  Let $\zeta_{0} : =0$ and for $n \ge 0$, define $\zeta_{n+1} : = \inf \{ t \ge \zeta_{n}: |\Delta X(t)| = | X(t) - X(t-) | > 0\}$.
 \begin{Proposition}\label{prop-no-common-jump}
 Under $\wdh \P$, $\{\zeta_{n}: n \ge 1\} $ and $\{\tau_{n}: n \ge 1 \}$ are mutually disjoint with probability $1$.
\end{Proposition}

\begin{proof}
We first notice that $\{\zeta_{n}: n \ge 1 \} = \bigcup_{m=1}^{\infty} \{ \zeta_{n}^{(1/m)}: n \ge 1 \}$ and hence
\begin{displaymath}
\{\zeta_{n}: n \ge 1 \} \cap \{\tau_{n}: n \ge 1 \} = \bigcup_{m=1}^{\infty} \{ \zeta_{n}^{(1/m)}: n \ge 1 \} \cap \{\tau_{n}: n \ge 1 \}.
\end{displaymath}
Moreover,  for each $m =1, 2, \dots$, since  \eqref{eq-nu-bdd} holds with $\e = \frac{1}{m}$,   Lemma \ref{lem-no-common-jumps} implies that
\begin{equation}\label{eq-pr=0}
\wdh\P \set{ \{ \zeta_{n}^{(1/m)}: n \ge 1 \} \cap \{\tau_{n}: n \ge 1 \}  \not= \emptyset} =0.
\end{equation}
Therefore we deduce
\begin{align*}
 \wdh\P\set{\{\zeta_{n}: n \ge 1 \} \cap \{\tau_{n}: n \ge 1 \} \not= \emptyset}
    & \le \sum_{m=1}^{\infty} \wdh\P \set{\{ \zeta_{n}^{(1/m)}: n \ge 1 \} \cap \{\tau_{n}: n \ge 1 \}  \not= \emptyset} =0.  \end{align*} This completes the proof.
\end{proof}

By virtue of $M_{t}$ and $\wdh{\bP}$, we can construct another probability measure $\bP$ on
$D([0,\infty),\rr^{d} \times \ss)$ such that $\bP$ is a solution to the martingale problem for the operator $\A$.

\begin{Theorem} \label{thm-general}
For any given $(x,k)\in \rr^{d}\times \ss$, there exists a unique martingale solution $\bP^{(x,k)}$ on $D([0,\infty),\rr^{d}
\times \ss)$ for the operator $\A$ starting from $(x,k)$.
\end{Theorem}

\begin{proof} For each $t\ge 0$ and each $A\in \F_{t}$, define
\begin{equation}\label{eq-P-general}
\P_{t}^{(x,k)}(A)=\int_{A}M_{t}(X(\cdot), \La(\cdot))\,\d\wdh{\P}^{(x,k)}.
\end{equation} Thanks to Lemma  \ref{Mtmartingale},   the family of probability measures $\{\bP_{t}^{(x,k)}\}_{t\ge 0}$ is consistent in the sense that if $0 \le t_{1} \le t_{2}$ and $A \in \F_{t_{1}}$,  then $\P_{t_{2}}^{(x,k)}(A)=  \P_{t_{1}}^{(x,k)}(A)$.
Thus by Tulcea's extension theorem (see, e.g., \cite[Theorem
1.3.5]{Stroock-V}), there exists a unique probability measure
$\bP^{(x,k)}$ on $(\Omega,\F)$ which coincides with
$\bP_{t}^{(x,k)}$ on $\F_{t}$ for all $t\ge 0$. Moreover, we will
prove that the $\bP$ is the desired martingale solution for the
operator $\A$ staring from $(x,k)$.  To do so, analogously to the proof of Lemma 4.2 in \cite{Xi-09}, we first prove that
for each function $f \in C^{\infty}_{c}(\rr^{d}\times \ss)$,
$\bigl(M_{t}M_{t}^{(f)}, \F_{t}, \wdh{\bP}\bigr)$ is a martingale,
where $M_{t}^{(f)}$ is defined in \eqref{martingale1}. In fact,  using integration by parts, we derive that
\begin{equation}\label{eq-Mt-Mt-f}
\begin{aligned}
M_{t}M_{t}^{(f)} &=\displaystyle\int_{0}^{t} M_{s-}^{(f)}\d  M_{s}+\int_{0}^{t} M_{s-} \d{\wdh M}_{s}^{(f)}  \\
 & \quad  + \displaystyle\int_{0}^{t} M_{s-} \bigl(\d M_{s}^{(f)}-\d{\wdh M}_{s}^{(f)}\bigr)
 + \displaystyle \sum_{s \le t}
\bigl(M_{s}-M_{s-}\bigr)\bigl(M_{s}^{(f)}-M_{s-}^{(f)}\bigr),
\end{aligned}
\end{equation} where $\wdh M_{t}^{(f)}$ is defined in \eqref{eq-Mt-hat}.
Using \eqref{martingale1}, \eqref{Mt},  and  Proposition \ref{prop-no-common-jump}, we can compute \begin{align*}
 \sum_{s \le t} & \bigl(M_{s}-M_{s-}\bigr)\bigl(M_{s}^{(f)}-M_{s-}^{(f)}\bigr) \\
    & = \sum_{s \le t}  (M_{s} - M_{s-}) [f(X(s),\La(s)) - f(X(s),\La(s-))]   \\
    &  = \int_{[0,t] \times \ss}M_{s-} \left( \frac{M_{s}}{M_{s-}} -1\right)  [f(X(s), l) - f(X(s), \La(s-))] n(\d s, \d l)\\
    & = \int_{[0,t] \times \ss}M_{s-} \left( q_{\La(s-)l} (X(s)) -1\right)  [f(X(s), l) - f(X(s), \La(s-))] n(\d s, \d l).
\end{align*}
On the other hand,
\begin{align*}
\lefteqn{  \int_{0}^{t} M_{s-} \bigl(\d M_{s}^{(f)}-\d{\wdh M}_{s}^{(f)}\bigr)} \\   & \
 \   = -  \int_{0}^{t} M_{s-}  \sum_{l\in \ss}  \bigl( q_{\La(s-)l} (X(s)) -1\bigr)\! \bigl[  f(X(s),l) - f(X(s),\La(s-))\bigr] \d s.
\end{align*}
Combining the last two displayed equations, and using the observations concerning the martingale measure $\wdt n(\cdot, \cdot)$ in the proof of Lemma \ref{Mtmartingale}, we obtain
\begin{align*}
  \lefteqn{\int_{0}^{t} M_{s-} \bigl(\d M_{s}^{(f)}-\d{\wdh M}_{s}^{(f)}\bigr)
+  \sum_{s \le t} \bigl(M_{s}-M_{s-}\bigr)\bigl(M_{s}^{(f)}-M_{s-}^{(f)}\bigr)} \\
 & \ \ = \int_{[0,t] \times \ss}M_{s-} \left( q_{\La(s-)l} (X(s)) -1\right)  [f(X(s), l) - f(X(s), \La(s-))] \wdt n(\d s, \d l). \end{align*}
 Then upon plugging the above equation into \eqref{eq-Mt-Mt-f}, it follows that
 \begin{equation}
\label{eq-Mt-Mt-f-mg}\begin{aligned}
M_{t}M_{t}^{(f)} &=\displaystyle\int_{0}^{t} M_{s-}^{(f)}\d  M_{s}+\int_{0}^{t} M_{s-} \d{\wdh M}_{s}^{(f)} \\
                          & \quad + \int_{[0,t] \times \ss}M_{s-} \left( q_{\La(s-)l} (X(s)) -1\right)  [f(X(s), l) - f(X(s), \La(s-))] \wdt n(\d s, \d l).
\end{aligned}\end{equation}
We have shown respectively in Theorem \ref{special}  and Lemma \ref{Mtmartingale} that $\wdh M_{\cdot}^{(f)}$ and $M_{\cdot}$ are martingales under  the measure $\wdh \P^{(x,k)}$. Also recall from the proof of Lemma \ref{Mtmartingale} that $\wdt n(\cdot, \cdot)$ is a martingale measure on $[0,\infty)\times \ss$ under $\wdh \P^{(x,k)}$. Thus in view of \eqref{eq-Mt-Mt-f-mg}, we conclude immediately that $M_{t}M_{t}^{(f)}$ is a martingale under $\wdh \P^{(x,k)}$.

We now
prove that for each function $f \in C^{\infty}_{c}(\rr^{d}\times
\ss)$, $\bigl(M_{t}^{(f)}, \F_{t}, \bP^{(x,k)} \bigr)$ is a martingale. Indeed, for any given $0\le s<t$ and any
given $A\in \F_{s}$, we have
$$\int_{A}M_{t}^{(f)}\d\P^{(x,k)}=\int_{A}M_{t}M_{t}^{(f)}\d\wdh{\P}^{(x,k)}
=\int_{A}M_{s}M_{s}^{(f)}\d\wdh{\P}^{(x,k)}=\int_{A}M_{s}^{(f)}\d\P^{(x,k)},$$ where the second equality follows from the martingale property of $\bigl(M_{t}M_{t}^{(f)}, \F_{t},
\wdh{\bP}^{(x,k)}\bigr)$, while the first   and the third equalities hold true since $\P^{(x,k)}$  coincides with    the probability measure   $\P_{t}^{(x,k)}$ given in \eqref{eq-P-general}.
This shows that $\P^{(x,k)}$ is a martingale solution for the operator $\A$ starting from $(x,k)$.

It remains to show that any martingale solution $\wdt \P$ for the operator $\A$ starting from $(x,k)$ must agree with $\P^{(x,k)}$ and therefore establishing the desired uniqueness.
From \cite{Wang-14}, for any martingale solution
$\bP^{(x,k)}$ to the operator $\A$, we have
$$\bP^{(x,k)}\bigl(\La(\tau_{1})\in
\ss\setminus\{k\}|\F_{\tau_{1}-}\bigr)=-\sum_{l\in
\ss\setminus\{k\}\ss\setminus\{k\}}\frac{q_{kl}}{q_{kk}}\bigl(X(\tau_{1}-)\bigr)=1.$$ Then the uniqueness can be established by using a similar  argument as that in the proof of Theorem \ref{special}.
\end{proof}

\begin{Remark}\label{rem-strong-Markov}
Thanks to Theorem \ref{thm-general}, the martingale problem for the operator $\A$ defined in \eqref{A} with any initial condition $(x,k) \in \R^{d}\times \ss$ is well-posed. Thus the process $(X,\La)$ is strong Markov.
\end{Remark}

\section{Feller Property}\label{sect-Feller}

We proved in Theorem \ref{thm-general} that  the martingale problem for the operator $\A$ defined in \eqref{A} is well-posed. Consequently for any $(x,k)$, there exists a unique probability measure $\P$ on $\Omega=D([0,\infty),\R^{d}\times \ss)$ under which  the coordinate process $(X(t),\La(t))$ satisfies $\P\{(X(0),\La(0)) =(x,k)\} =1$ and that for any $f\in C_{c}^{\infty}(\R^{d}\times \ss)$, the process  $M_{t}^{f}$ defined in \eqref{martingale1} is an $\{\F_{t}\}$-martingale.
In this section, we will prove that in the  probability space $(\Omega, \F, \P)$,  the process $(X(t),\La(t))$ possesses   the Feller property
  under the following  conditions.

\begin{Assumption} \label{FP1}
Assume that there exist a positive  constant $H$ and  a nondecreasing and concave function $\rho: [0,\infty) \mapsto [0,\infty)$ satisfying $\rho(r) > 0$ for $r > 0$ and  \begin{equation}
\label{eq-rho-sec-2}\int_{0+}\frac{ \d r}{ \rho(r)} = \infty,
\end{equation}
such that for all  $k \in \ss$ and $  x,z\in \R^{d}$,
\begin{equation}\label{(FP1)}
\|\sigma(x,k)-\sigma(z,k)\|^2 +2 \langle x-z, b(x,k)-b(z,k)\rangle \le
H  |x-z| \rho(|x-z|),\end{equation} and
\begin{equation}\label{(FP2)}
\int_{\R^{d}_{0}} |u|\|\nu(x,k,\cdot)-
\nu(z,k,\cdot)\|(\d u) \le H    \rho(|x-z|),  \end{equation}
  where $\sigma(x,k)\in \R^{d\times d}$ satisfies $\sigma(x,k)\sigma(x,k)^{T}= a(x,k)$, and
$\|\cdot\|$ denotes the Hilbert-Schmidt norm for matrices or the
total variation norm for signed measures. Here and   below, $^{T}$ denotes the transpose of a vector or matrix.
\end{Assumption}

\begin{Assumption} \label{FP4}
 Assume that
\begin{equation}\label{(FP12)}
|q_{kl}(x) - q_{kl}(z)| \le H |x-z|\end{equation} for all $x,z \in
\rr^{d}$ and $k \ne l\in \ss$, where constant $H>0$ is the same as that in
Assumption~\ref{FP1} without loss of generality.\end{Assumption}

\begin{Remark} \label{FP2}
For existence of a square root $\sigma(x,k)$ of $a(x,k)$ such as in
Assumption~\ref{FP1} and the equivalence of different choices of the
square root, we refer the reader to \cite{Stroock-V} for the details.
  Some common functions  satisfying the conditions in Assumption  \ref{FP1}   include $\rho(r) = r$ and concave and increasing  functions such as $\rho(r) = r \log(1/r)$,  $\rho(r) = r \log( \log(1/r))$, and $\rho(r) = r \log(1/r)  \log( \log(1/r))$ for $r \in (0,\delta)$ with $\delta > 0$ small enough. \end{Remark}
The main result of this section is:
\begin{Theorem} \label{thm-Feller} Suppose that Assumptions \ref{assump-Lhasonlyone},
\ref{FP1} and \ref{FP4} hold. Then the process $(X,\Lambda)$ has
 Feller property.
\end{Theorem}

 Let us first briefly describe our strategy toward the proof of Theorem \ref{thm-Feller}. We first use 
the coupling method  to show in Lemma \ref{Lem-FP3} that   the L\'evy type process $\wdt X^{(k)}$ corresponding to the operator $\LL_{k}$ of \eqref{L} has Feller property under Assumptions \ref{assump-Lhasonlyone} (i) and
\ref{FP1}. Lemma \ref{FP5} further establishes the Feller property  for the killed  L\'evy type process $X^{(k)}$ under Assumption \ref{FP4}. Next we show in Lemma \ref{lem-resolvant-series} that the resolvent of $(X,\La)$ can be represented by a  series of the resolvents of the killed processes $X^{(k)}, k \in \ss$. This representation  further  helps us to establish \eqref{(FP22)}. Finally we use \eqref{(FP22)}  together with Proposition 6.1.1 in \cite{MeynT-93} to derive the Feller property for the process $(X,\La)$.

 \begin{Remark}\label{rem-comparison-Wang-10}The recent paper \cite{Wang-10} 
 also establishes the Feller property for the L\'evy type process $\wdt X^{(k)}$ under a different set of conditions. In particular,   the L\'evy type kernel is assumed to have a certain  representation in \cite{Wang-10}. By contrast,  our goal is to 
 establish the Feller property for the two-component process $(X,\Lambda)$ under Assumptions \ref{assump-Lhasonlyone}, \ref{FP1} and \ref{FP4}. This is achieved by establishing the Feller property for   the L\'evy type process $\wdt X^{(k)}$ as well as the killed L\'evy type process $X^{(k)}$ under these assumptions.  
 It is worth pointing out that Lemma \ref{Lem-FP3} below indicates that Assumptions \ref{assump-Lhasonlyone} (i) and \ref{FP1} are sufficient conditions for   the Feller property  for the L\'evy type process $\wdt X^{(k)}$. These assumptions,
 in particular, Assumption \ref{FP1},  seem more direct and easier to verify in some sense compared with those in \cite{Wang-10}.
\end{Remark}

  Recall that for each   $k \in \ss$ and $x\in \R^{d}$, Assumption \ref{assump-Lhasonlyone} guarantees that the operator $\LL_{k}$ of \eqref{L}  uniquely determines  a L\'{e}vy type process $\wdt {X}^{(k)(x)}$ with initial condition $\wdt {X}^{(k)(x)}(0)=x$.
   Next 
 we kill the
  process $\wdt{X}^{(k)(x)}$ at  rate $(-q_{kk})$:
\begin{equation}\label{kp1}
\begin{array}{ll}
\E_{k}[f(X^{(k)(x)}(t))]\ad =\displaystyle
\E_{k}[f(\wdt{X}^{(k)(x)}(t)); t<\tau]\\ \ad \displaystyle
=\E_{k}\biggl[\exp\biggl\{\int_{0}^{t}q_{kk}(\wdt{X}^{(k)(x)}(s))\d s\biggr\}f(\wdt{X}^{(k)(x)}(t))\biggr],
\end{array}
\end{equation} 
where $\tau:=\inf\{ t\ge 0: \La(t) \not=\La(0)\}$. Equivalently, the killed L\'{e}vy type
$X^{(k)(x)}$ can be defined as $X^{(k)(x)}(t)=\wdt{X}^{(k)(x)}(t)$ if $t<\tau$
and $X^{(k)(x)}(t)=\partial$ if $t\ge \tau$, where $\partial$ is a
cemetery point added to $\R^{d}$.
 Moreover, we denote the transition probability families of the L\'{e}vy type process $\wdt {X}^{(k)}$ and the killed L\'{e}vy type process $X^{(k)}$ by $\{\wdt {P}^{(k)}(t,x,A): t \ge 0, x \in \rr^{d}, A \in \B(\rr^{d})\}$ and $\{P^{(k)}(t,x,A): t \ge 0, x \in \rr^{d}, A \in\B(\rr^{d})\}$, respectively.

For an arbitrarily fixed $k \in \ss$, we now construct a coupling of
the L\'{e}vy type process $\wdt {X}^{(k)}$. To this end, we need only to
construct a coupling for  its generator ${\LL}_{k}$. For $x,z \in
\R^{d}$, set
$$a(x,z,k)=\left(\begin{array}{cc}
a(x,k) & \sigma (x,k) \sigma (z,k)^T \\
\sigma (z,k) \sigma (x,k)^T & a(z,k)
\end{array}\right),\quad
b(x,z,k)=\left(\begin{array}{c}
b(x,k)\\
b(z,k) \end{array}\right).$$  Obviously, $a(x,z,k)$ is nonnegative
definite for all $x,z \in \rr^{d}$. For $h(x,z) \in
C^{2}_{0}(\rr^{d} \times \rr^{d})$, set
\begin{equation}\label{eq-Omega-d-defn}
\wdt {\Omega}_{d}(k)h(x,z)=\frac
{1}{2}\hbox{tr}\bigl(a(x,z,k)\nabla^{2}h(x,z)\bigr)+\langle
b(x,z,k), \nabla h(x,z)\rangle,
\end{equation} which is a coupling of the
diffusion part in the generator ${\LL}_{k}$ defined in \eqref{L} (refer to \cite{ChenLi-89}).
Next, for $h(x,z) \in C^{2}_{0}(\rr^{d} \times \rr^{d})$, set
\begin{equation}\label{eq-Omega-j-defn}
\begin{aligned}\displaystyle &\wdt {\Omega}_{j} (k)h(x,z) \\ & =
\int
[h(x+u,z)-h(x,z)- \langle \nabla_{x } h(x,z),   u\rangle {\mathbf{1}}_{B(0, \e_{0})}(u)]\bigl(\nu(x,k,\d u)-\nu(z,k,\d u)\bigr)^{+}\\
& \quad + \int  [h(x,z+u)-h(x,z)- \langle \nabla_{z} h(x,z), u \rangle {\mathbf{1}}_{B(0, \e_{0})}(u)]\bigl(\nu(z,k,\d u)-\nu(x,k,\d u)\bigr)^{+}\\
& \quad + \int
[h(x+u,z+u)-h(x,z)-  \langle \nabla_{x } h(x,z),   u\rangle {\mathbf{1}}_{B(0, \e_{0})}(u)\\
& \qquad\qquad-\langle \nabla_{z} h(x,z), u \rangle {\mathbf{1}}_{B(0, \e_{0})}(u)] \bigl(\nu(x,k,(\cdot))\wedge \nu(z,k,(\cdot))\bigr)(\d u), \end{aligned}
\end{equation}   where $(\nu(x,k,\cdot)-\nu(z,k,\cdot))^{+} =\sup\{\nu(x,k,A)-\nu(z,k,A): A\in \B(\R_{0}^{d})\}$ and $(\nu(z,k,\cdot)-\nu(x,k,\cdot))^{+}$ is defined in a similar fashion.   Note that the operator $\wdt {\Omega}_{j} (k)$ defined in \eqref{eq-Omega-j-defn}  is a coupling of the jump
part in the generator ${\LL}_{k}$ defined in \eqref{L}.
 Finally,
combining the two couplings together, we get a coupling
$\wdt {{\LL}}_{k}$ of the generator ${\LL}_{k}$ as follows:
\begin{equation}\label{(FP5)}
\wdt {{\LL}}_{k}
h(x,z)=\wdt {\Omega}_{d}(k)h(x,z)+\wdt {\Omega}_{j}(k)h(x,z),
\end{equation} for $h(x,z) \in
C^{2}_{0}(\rr^{d} \times \rr^{d})$.

To proceed, we now introduce the Wasserstein metric
between two probability measures as follows. For two probability
measures $P_1$ and $P_2$ on $(\rr^{d},\B(\rr^{d}))$, define
$$W\bigl(P_{1}, P_{2} \bigr)
= \inf_{\wdt {P}} \int |x-z| \wdt {P} (\d x,\d z),$$ where $\wdt {P}$ varies
over all coupling probability measures with marginals $P_{1}$ and
$P_{2}$; that is, 
\begin{displaymath}
\wdt P(A\times \R^{d}) = P_{1}(A), \text{ and } \wdt P(\R^{d}\times A) = P_{2}(A), \text{ for any }A \in \B(\R^{d}).
\end{displaymath}

\begin{Lemma} \label{Lem-FP3}
Suppose that Assumptions~\ref{assump-Lhasonlyone} (i) and \ref{FP1} hold. For each $k
\in \ss$, the L\'{e}vy type process $\wdt {X}^{(k)}$ generated by the
L\'{e}vy type operator ${\LL}_{k}$ defined in \eqref{L} has Feller
property.
\end{Lemma}

\begin{proof} For an arbitrarily fixed $k \in \ss$, we need only to
prove that for any $t>0$, $x,z \in \rr^{d}$,
$\wdt {P}^{(k)}(t,x,\cdot)$ converges weakly to
$\wdt {P}^{(k)}(t,z,\cdot)$ as $x \to z$. By virtue  of
Theorem 5.6 in \cite{Chen04}, it suffices to prove that
\begin{equation}\label{(FP6)}
W \bigl(\wdt {P}^{(k)}(t,x,\cdot), \wdt {P}^{(k)}(t,z,\cdot)\bigr) \to 0
\qquad \hbox{as} \quad  x \to z. \end{equation} We use the
coupling $\wdt {{\LL}}_{k}$ constructed in \eqref{(FP5)} to establish \eqref{(FP6)}. Let
$(\wdt {X}^{(k)}, \wdt {Z}^{(k)})$ denote the coupling process
corresponding to the coupling generator $\wdt {{\LL}}_{k}$.   Also let   $\bP_{k}$ denote the distribution of
$(\wdt {X}^{(k)},\wdt {Z}^{(k)})$ and   $\E_{k}$   the
corresponding expectation with a slight  abuse of notation.
 By Assumption~\ref{assump-Lhasonlyone} we readily know that
the coupling process $(\wdt {X}^{(k)}, \wdt {Z}^{(k)})$ is
non-explosive. Similarly to the proof of Theorem 2.3 in
\cite{ChenLi-89}, set
\begin{align*}
T_R &: = \inf \{t \ge 0: |\wdt {X}^{(k)}(t)|^2 +|\wdt {Z}^{(k)}(t)|^2 > R \}.
\end{align*}

  Thanks to the assumptions imposed on the function $\rho$, we can find a strictly decreasing sequence $\{a_{n}\}\subset (0, 1]$ with $a_{0} =1$, $\lim_{n\to\infty} a_{n} =0$ and $\int_{a_{n}} ^{a_{n-1}} \rho^{-1} (r) \d r = n$ for every $n \ge 1$. For each $n \ge 1$, there exists a continuous function $\rho_{n}$ on $\R$ with support in $(a_{n}, a_{n-1}) $ so that  $0 \le \rho_{n}(r) \le 2 n^{-1}\rho^{-1}(r) $ holds for every $r > 0$, and $\int_{a_{n}}^{a_{n-1}} \rho_{n}(r) \d r =1$.

Now consider the sequence of functions
\begin{equation}
\label{eq-psi-n-sec-2}
\psi_{n}(r) : = \int_{0}^{|r|}\int_{0}^{y} \rho_{n}(u) \d u \d y, \quad r\in \R, n \ge 1.
\end{equation}
We can immediately verify that $\psi_{n}$ is even and continuously differentiable, with $|\psi_{n}'(r)| \le 1$ and $\lim_{n\to\infty} \psi_{n}(r) = |r|$ for $r\in \R$. Furthermore, for each $r > 0$, the sequence $\{\psi_{n}(r) \}_{n\ge 1}$ is nondecreasing. Note also that for each $n\in \N$,  $\psi_{n}, \psi_{n}'$ and $\psi_{n}''$ all vanish  on the interval $(-a_{n}, a_{n})$.


 For any $x,z \in \rr^{d}$, set
\begin{align*}
A(x,z,k)& =a(x,k)+a(z,k)-2 \sigma (x,k) \sigma (z,k)^T,\\
\wdh {B}(x,z,k)& =\langle x-z, b(x,k)-b(z,k) \rangle, \end{align*}
 and
\begin{displaymath}
\lbar{A} (x,z,k)   = \langle x-z, A(x,z, k) (x-z) \rangle / |x-z|^{2}.
\end{displaymath}
Then as in the proof of Theorem 3.1 in \cite{ChenLi-89}, we can verify directly that
\begin{equation}
\label{eq-Omega-dk-psi n expression}
\begin{aligned}
2 \wdt \Omega_{d}(k) \psi_{n}(|x-z|) & = \psi_{n}''(|x-z|)\lbar A (x,z,k)  \\ & \qquad + \frac{ \psi_{n}'(|x-z|) }{|x-z|} \big[ \tr(A(x,z,k))-\lbar A(x,z,k) + 2 \wdh B(x,z,k)\big].
\end{aligned}
\end{equation}
Note that
$\tr (  A(x,z,k))=\|\sigma (x,k)-\sigma (z,k)\|^{2}$ and hence we obtain from    \eqref{(FP1)} that  $$\tr A(x,z,k)  + 2 \wdh B(x,z,k) \le H |x-z|\rho(|x-z|).$$   On the other hand, using \eqref{(FP1)} again,
\begin{displaymath}
\lbar A(x,z,k) = \frac{\langle x-z, (\sigma(x,k)-\sigma(z,k)) (\sigma(x,k)-\sigma(z,k))^{T} (x-z)\rangle}{|x-z|^{2}} \le H |x-z|\rho(|x-z|).
\end{displaymath} Thanks to the construction of $\psi_{n}$, we have  $0\le \psi_{n}'(r) \le 1$ and
$\psi_{n}''(r) = \rho_{n}(r) \le \frac{2}{n\rho(r)} I_{(a_{n},a_{n-1})}(r)$ for all $r\ge 0$.
Putting the above estimates  into \eqref{eq-Omega-dk-psi n expression},   it then follows that
\begin{align}
\label{eq1-Omega-d}
\nonumber\wdt \Omega_{d}(k) \psi_{n}(|x-z|) &   \le \frac{1}{2} \psi_{n}''(|x-z|) H  |x-z| \rho(|x-z|) + \frac{1}{2} \psi_{n}'(|x-z|) H  \rho(|x-z|) \\
    \nonumber      & \le  \frac{H}{n}  |x-z|  I_{(a_{n},a_{n-1})}(|x-z| )+  \frac{1}{2}   H \rho(|x-z|)\\
          & \le   \frac{Ha_{n-1}}{n} +  \frac{1}{2}   H \rho(|x-z|).
 \end{align}

By virtue of the mean value theorem and the fact that $|\psi_{n}'| \le 1$, we have \begin{displaymath} \psi_{n}(|x+u-z|) - \psi_{n}(|x-z|) \le \abs{|x+u-z| - |x-z|} \le |u|, \end{displaymath} and   \begin{displaymath}
\abs{\langle \nabla_{x} \psi_{n}(|x-z|), u \rangle {\mathbf{1}}_{B(0, \e_{0})}(u)} \le |u|.  
\end{displaymath}
Then it follows that    \begin{align*}
\int & \big( \psi_{n}(|x+u-z|) - \psi_{n}(|x-z|)- \langle \nabla_{x} \psi_{n}(|x-z|), u \rangle {\mathbf{1}}_{B(0, \e_{0})}(u) \big) \big( \nu (x,k,\d u) - \nu (z,k,\d u) \big)^{+}  \\
& \le 2 \int |u| \big( \nu (x,k,\d u) - \nu (z,k,\d u) \big)^{+}.
\end{align*}
Similarly, we have   \begin{align*} \int & \big( \psi_{n}(|x-z-u|) - \psi_{n}(|x-z|)- \langle \nabla_{z} \psi_{n}(|x-z|), u \rangle {\mathbf{1}}_{B(0, \e_{0})}(u) \big) \big( \nu (z,k,\d u) - \nu (x,k,\d u) \big)^{+}  \\
& \le 2 \int |u| \big( \nu (x,k,\d u) - \nu (z,k,\d u) \big)^{+}. \end{align*}
Note that $\nabla_{x} \psi_{n} (|x-z|) = - \nabla_{z} \psi_{n} (|x-z|)$. Thus
\begin{equation*}\label{eq-x^z-int=0}\begin{aligned}
   \int \big[ &  \psi_{n}(|x+u-z-u|) - \psi_{n}(|x-z|) -\langle \nabla_{x} \psi_{n}(|x-z|), u \rangle {\mathbf{1}}_{B(0, \e_{0})}(u)  \\  & -  \langle \nabla_{z} \psi_{n}(|x-z|), u \rangle {\mathbf{1}}_{B(0, \e_{0})}(u) \big] \big( \nu (x,k,\d u)  \wedge \nu (z,k,\d u) \big) = 0. \end{aligned}\end{equation*}
Then, using the definition of $\wdt \Omega_{j}(k)$ in \eqref{eq-Omega-j-defn}, we obtain
\begin{align}
\label{eq-Omega-j-est}
\nonumber \wdt \Omega_{j}& (k) \psi_{n} (|x-z|) \\
\nonumber & \le 2  \int |u| \big( \nu (x,k,\d u) - \nu (z,k,\d u) \big)^{+} + 2 \int |u|  \big( \nu (z,k,\d u) - \nu (x,k,\d u) \big)^{+} \\ &   \le 2 \int |u|  \| \nu (x,k, \cdot) - \nu (z,k, \cdot)  \| (\d u) \le 2 H \rho(|x-z|),
 \end{align}  where the last inequality follows from \eqref{(FP2)}.

A combination of \eqref{eq1-Omega-d} and \eqref{eq-Omega-j-est} yields
\begin{displaymath}
\wdt \LL_{k} \psi_{n}(|x-z|) =  \wdt \Omega_{d}(k) \psi_{n}(|x-z|) + \wdt \Omega_{j}(k) \psi(|x-z|) \le  \frac{H a_{n-1}}{n} +  3  H\rho( |x-z|),\ \forall x,z\in \R^{d}.
\end{displaymath}
Now we apply It\^o's formula to the process $\psi_{n}(|\wdt X^{(k)(x)}(\cdot) - \wdt Z^{(k)(z)}(\cdot)|)$ to obtain
\begin{equation}\label{eq-psi-n-mean}\begin{aligned}
\E_{k}&\bigl [\psi_{n}(|\wdt X^{(k)(x)}(t\wedge T_{R} ) - \wdt Z^{(k)(z)}(t\wedge T_{R} )|)\bigr]\\ &  = \psi_{n}( |x - z|) + \E_{k} \biggl[ \int_{0}^{t\wedge T_{R} } \wdt \LL_{k} \psi_{n} (|\wdt X^{(k)(x)}(s) - \wdt Z^{(k)(z)}(s)|) \d s  \biggr]\\
  & \le \psi_{n}( |x - z|)  + \frac{Ha_{n-1}t}{n} + 3H  \E_{k}\biggl[ \int_{0}^{t\wedge T_{R} }\rho\bigl( |\wdt X^{(k)(x)}(s) - \wdt Z^{(k)(z)}(s)|\bigr) \d s \biggr].
\end{aligned}\end{equation}

Recall that  $\psi_{n}(|x|) \uparrow |x|$ and $a_{n}\to 0$ as $n \to \infty$. Therefore, passing to the limit as $n \to \infty$ on both sides of \eqref{eq-psi-n-mean}, it follows from the monotone convergence theorem that
\begin{align*}
\E_{k} & \bigl[|\wdt X^{(k)(x)}(t\wedge T_{R} ) - \wdt Z^{(k)(z)}(t\wedge T_{R} )|\bigr] \\ &  \le |x-z| +   3 H  \E_{k}\biggl[ \int_{0}^{t\wedge T_{R} } \rho\big(|\wdt X^{(k)(x)}(s) - \wdt Z^{(k)(z)}(s)|\bigr) \d s \biggr] \end{align*}
Furthermore, passing to the limit as $R \to \infty$, we have by Fatou's lemma and the monotone convergence theorem that
\begin{align}\label{eq-u(t)<=v(t)}
\nonumber\E_{k}  \bigl[|\wdt X^{(k)(x)}(t  ) - \wdt Z^{(k)(z)}(t  )|\bigr]
 & \le  |x-z| +  3 H  \E_{k}\biggl[ \int_{0}^{t  } \rho\big(|\wdt X^{(k)(x)}(s) - \wdt Z^{(k)(z)}(s)|\bigr) \d s \biggr]\\
\nonumber & \le  |x-z| + 3  H  \E_{k}\biggl[ \int_{0}^{t}\rho\big( |\wdt X^{(k)(x)}(s  ) - \wdt Z^{(k)(z)}(s )|\big) \d s \biggr]  \\
 & \le  |x-z| + 3  H   \int_{0}^{t}\rho\Big(\E_{k}\big[ |\wdt X^{(k)(x)}(s ) - \wdt Z^{(k)(z)}(s)|\big]\Bigr )\d s,
\end{align} where the last inequality follows from Fubini's theorem and Jenson's inequality. Denote $u(t): = \E_{k}  \bigl[|\wdt X^{(k)(x)}(t  ) - \wdt Z^{(k)(z)}(t)|\bigr]$ and $v(t): = |x-z| + 3 H \int_{0}^{t} \rho(u(s))\d s$. Then by \eqref{eq-u(t)<=v(t)}, we have $0 \le u(t) \le v(t).$  Define $G(r): = \int_{1}^{r}\frac{\d s}{\rho(s)}$ for $r >0$. Then $G$ is nondecreasing and satisfies $\lim_{r\downarrow 0} G(r) = -\infty$ thanks to \eqref{eq-rho-sec-2}. In addition, we have
 \begin{align*}
G(u(t)) & \le       G(v(t)) = G(|x-z |) + \int_{0}^{t} G'(v(s)) v'(s) \d s \\
         & = G(|x-z |) + 3 H \int_{0}^{t}  \frac{  \rho(u(s))}{\rho(v(s))}\d s
      \le G(| x-z |) + 3H  t,
\end{align*} where the last inequality follows from the assumption that $\rho $ is nondecreasing. Let also $G^{-1}(r) : = \inf\{s \ge 0: G(s) > r\}$ for $r\in \R$.  Then $G^{-1}$ is   nondecreasing and satisfies $\lim_{r\to -\infty} G^{-1}(r) = 0$. Furthermore, we have
\begin{equation}\label{(FP11)}0 \le u(t)=\E_{k}  \bigl[|\wdt X^{(k)(x)}(t  ) - \wdt Z^{(k)(z)}(t)|\bigr] \le G^{-1}(G(| x-z |) + 3H  t).\end{equation}
In particular, when   $  |x-z| \to 0$,
we see that the right most expression of \eqref{(FP11)} converges to $0$
and so does $u(t) $.
This implies \eqref{(FP6)} and hence completes the proof. \end{proof}

\begin{Lemma} \label{FP5}
Suppose that Assumptions~\ref{assump-Lhasonlyone}, \ref{FP1} and \ref{FP4} hold. For
each $k \in \ss$, the killed L\'{e}vy type process $X^{(k)}$
introduced in \eqref{kp1} has Feller property.
\end{Lemma}

\begin{proof} For an arbitrarily fixed $k \in \ss$, we need only to
prove that for any given $t>0$ and $f \in C_{b}(\rr^{d})$,
\begin{equation}\label{(FP13)}
\begin{aligned}
\ad \bigl|\E_{k} \bigl[f(X^{(k)(x)}(t))\bigr]-\E_{k} \bigl[f(X^{(k)(z)}(t))\bigr]\bigr|\\
\aad =\biggl| \E_{k}\biggl[f(\wdt {X}^{(k)(x)}(t))\exp
\biggl\{\int_{0}^{t}q_{kk}(\wdt {X}^{(k)(x)}(s))\d s\biggr\}\biggr]  \\
  \aad \qquad
- \E_{k}\biggl[f(\wdt {X}^{(k)(z)}(t))\exp
\biggl\{\int_{0}^{t}q_{kk}(\wdt {X}^{(k)(z)}(s))\d s\biggr\}\biggr]\biggr|
\end{aligned}
\end{equation} tends to zero as $|x-z| \to 0$.
Using the coupling process $(\wdt {X}^{(k)}, \wdt {Z}^{(k)})$ generated
by the coupling generator $\wdt {{\LL}}_{k}$ as in the proof of
Lemma~\ref{Lem-FP3}, we obtain that for any given $\varepsilon >0$, the
right-hand side of equality (\ref{(FP13)}) equals
\begin{align}\label{(FP14)}
\nonumber \ad \biggl| \E_{k}\biggl[f(\wdt {X}^{(k)(x)}(t))\exp
\biggl\{\int_{0}^{t}q_{kk}(\wdt {X}^{(k)(x)}(s))\d s\biggr\}\biggr]
\\
\nonumber \aad \qquad - \E_{k}\biggl[f(\wdt {Z}^{(k)(z)}(t))\exp
\biggl\{\int_{0}^{t}q_{kk}(\wdt {Z}^{(k)(z)}(s))\d s\biggr\}\biggr]\biggr|\\
\nonumber \aad \le\E_{k}  \biggl[ \biggl| f(\wdt {X}^{(k)(x)}(t))\exp
\biggl\{\int_{0}^{t}q_{kk}(\wdt {X}^{(k)(x)}(s))\d s\biggr\}
\\
\nonumber \aad \qquad \quad - f(\wdt {Z}^{(k)(z)}(t))\exp
\biggl\{\int_{0}^{t}q_{kk}(\wdt {Z}^{(k)(z)}(s))\d s \biggr\}\biggr|\biggr]\\
\aad \le \|f\| \, \E_{k}\biggl[\biggl|\exp
\biggl\{\int_{0}^{t}q_{kk}(\wdt {X}^{(k)(x)}(s))\d s\biggr\}- \exp
\biggl\{\int_{0}^{t}q_{kk}(\wdt {Z}^{(k)(z)}(s))\d s\biggr\}\biggr| \biggr]\\
\nonumber \aad \quad +2 \|f\| \,\E_{k} \biggl[\exp
\biggl\{\int_{0}^{t}q_{kk}(\wdt {Z}^{(k)(z)}(s))\d s\biggr\} \one_{\{ |f(\wdt {X}^{(k)(x)}(t))-f(\wdt {Z}^{(k)(z)}(t))| \ge
\varepsilon\}}\biggr] \\
\nonumber  \aad \quad +   \e \,\E_{k} \biggl[\exp
\biggl\{\int_{0}^{t}q_{kk}(\wdt {Z}^{(k)(z)}(s))\d s\biggr\} \one_{\{ |f(\wdt {X}^{(k)(x)}(t))-f(\wdt {Z}^{(k)(z)}(t))| <
\varepsilon\}}\biggr]
\\
\nonumber  \aad :=(\ref{(FP14)}.\hbox{I})+(\ref{(FP14)}.\hbox{II})
+(\ref{(FP14)}.\hbox{III}),
\end{align} where $\|f\|$ denotes the uniform (or supremum) norm of the function $f$. Noting
that $q_{kk} \le 0$ and the elementary inequality $|e^{-a} - e^{-b}| \le |a-b|$ for $a,b > 0$, we obtain from   \eqref{(FP12)} and \eqref{(FP11)} that
\begin{equation}\label{(FP15)}
\begin{array}{ll}
(\ref{(FP14)}.\hbox{I})\ad \le \|f\| \, \E_{k} \biggl[
\biggl|\int_{0}^{t}q_{kk}(\wdt {X}^{(k)(x)}(s))\d s-
\int_{0}^{t}q_{kk}(\wdt {Z}^{(k)(z)}(s))\d s\biggr|\biggr]\\
\ad \le (n_{0}-1)H \|f\| \, \int_{0}^{t}\E_{k}\biggl[ \Bigl|\wdt {X}^{(k)(x)}(s)-
\wdt {Z}^{(k)(z)}(s) \Bigr| \biggr]\d s\\
\ad \le    (n_{0}-1)H \|f\| \, \int_{0}^{t} G^{-1}(G(| x-z |) + 3H  s) \d s,                    
\end{array}
\end{equation} where $G$ and $G^{-1}$ are the functions defined in
the proof of Lemma \ref{Lem-FP3}. Since both $G$ and $G^{-1}$ are nondecreasing, for all $s\in [0,t]$ and $x,z\in \R^{d}$ with $|x-z| \le 1$, we have $0 \le G^{-1}(G(| x-z |) + 3H  s) \le G^{-1}(G(1) + 3 H t)= G^{-1}(3 Ht) $, which is integrable on the interval $[0,t]$. Thus it follows from the dominated convergence theorem and \eqref{(FP11)} that $(\ref{(FP14)}.\hbox{I})\to 0$  as $|x-z| \to 0$. 
Moreover, in view of (\ref{(FP11)}), we have
that $\wdt {X}^{(k)(x)}(t)$ converges to $\wdt {Z}^{(k)(z)}(t)$ in
probability $\bP_{k}$ as $|x-z| \to 0$. Thus, from the continuity of
$f$, 
we obtain that $f(\wdt {X}^{(k)(x)}(t))$ also converges to
$f(\wdt {Z}^{(k)(z)}(t))$ in probability $\bP_{k}$ as $|x-z| \to 0$.
Combining this with $q_{kk} \le 0$, we derive that
\begin{equation}\label{(FP16)}
(\ref{(FP14)}.\hbox{II}) \le 2 \|f\| \,
\bP_{k}\bigl(\bigl|f(\wdt {X}^{(k)(x)}(t))-f(\wdt {Z}^{(k)(z)}(t))\bigr|
\ge \varepsilon \bigr) \to 0
\end{equation} as $|x-z| \to 0$. Using the fact that $q_{kk} \le 0$ again, we see that  $(\ref{(FP14)}.\hbox{III})$
does not exceed $\varepsilon$; which can be arbitrarily small.
Combining this, \eqref{(FP14)},  \eqref{(FP15)} and  \eqref{(FP16)}
together, we conclude that the right-hand side of equality
(\ref{(FP13)}) tends to zero as $|x-z| \to 0$. The proof is
complete. \end{proof}

\begin{Lemma}\label{lem-CP1}
Let Z be the subprocess of $\wdt{Z}$ killed at the rate $q$ with
lifetime $\zeta$,  that is,
\begin{equation}\label{kp2}
\E[f(Z^{(x)}({t}))] =\displaystyle \E\bigl[f(\wdt{Z}^{(x)}(t)); t<\zeta\bigr]
=\E\biggl[\exp\biggl\{-\int_{0}^{t}q(\wdt{Z}^{(x)}(s))\d s\biggr\}f(\wdt{Z}^{(x)}(t))\biggr],
\end{equation}
where $\wdt{Z}$ is a right continuous strong Markov process,  $q \ge 0$ on $\rr^{d}$, and $f \in \B_{b}(\R^{d})$. Then for any nonnegative function  $\phi   $ on $\rr^{d}$
and constant $\al >0$, we have
\begin{equation}\label{eq:zeta}
\E[e^{-\al \zeta}\phi(Z^{(x)}(\zeta-))]=G_{\al}^{Z}(q\phi)(x),
\end{equation} where $\{G_{\al}^{Z}, \al >0\}$ denotes the resolvent
for the killed process $Z$.
\end{Lemma}

\begin{proof} By the   definition of the resolvent and
(\ref{kp2}), we get
$$\begin{array}{ll}
G_{\al}^{Z}(q\phi)(x)\ad =\E \biggl[\int_{0}^{\infty} e^{-\al
t}(q\phi)(Z^{(x)}(t))\d t\biggr]\\
\ad =\E \biggl[\int_{0}^{\infty} e^{-\al
t}(q\phi)(\wdt{Z}^{(x)}(t))\exp\biggl\{-\int_{0}^{t}q(\wdt{Z}^{(x)}(s))\d s\biggr\}\d t\biggr],
\end{array}$$ which by page 286 in  \cite{Sharpe-88} (putting
$m_t=\exp\{-\int_{0}^{t}q(\wdt{Z}(s))\d s\}{\mathbf{1}}_{(t<\zeta)}$
there) equals the left-hand side in (\ref{eq:zeta}). \end{proof}

  For each $k \in \ss$, let $\{G^{(k)}_\al, \al>0\}$
be the resolvent for the generator $\LL_k+q_{kk}$. Let us also denote by $\{G_{\alpha}, \alpha >0 \}$
the resolvent for the generator $\A$ defined in \eqref{A}.
  Let
$$G^{0}_{\al}=\left(\begin{array}{cccc}
G^{(1)}_{\al} & 0 & \cdots & 0\\
0 & G^{(2)}_{\al} & \cdots & 0\\
\vdots & \vdots & \ddots & \vdots\\
0 & 0 & \cdots & G^{(n_{0})}_{\al}
\end{array} \right) \ \hbox{and} \  Q^{0}(x)=Q(x)-\left(\begin{array}{cccc}
q_{11}(x) & 0 & \cdots & 0\\
0 & q_{22}(x) & \cdots & 0\\
\vdots & \vdots & \ddots & \vdots\\
0 & 0 & \cdots & q_{n_{0}n_{0}}(x)
\end{array} \right).$$

\begin{Lemma} \label{lem-resolvant-series}
Suppose that Assumption~\ref{assump-Lhasonlyone} holds. There exists a constant
${\al}_1>0$ such that for any ${\al}\ge {\al}_1$ and any
$f(\cdot,k)\in \B_{b}(\rr^{d})$ with $k \in \ss$,
\begin{equation}\label{(FP17)}
G_{\al}f=G^{0}_{\al}
f+\sum_{m=1}^{\infty}G^{0}_{\al}\bigl(Q^{0}G^{0}_{\al}\bigr)^{m}f.
\end{equation}
\end{Lemma}

 \begin{proof}
 Let $f(x,k)\geq
0$ on $\rr^{d} \times \ss$. Applying the strong Markov property at
the first switching time $\tau$ and recalling the construction of
$(X,\La)$, we obtain
\begin{align*} \displaystyle
G_\al f(x,k) & =\E_{x,k} \biggl[\int_0^\infty e^{-{\al} t}
f(X(t),\Lambda (t))\d t \biggr]\\
&=\E_{x,k} \biggl[\int_0^\tau e^{-{\al} t} f(X(t), k)\d t \biggr]+
\E_{x,k} \biggl[ \int_\tau^\infty e^{-{\al} t}
f(X(t), \Lambda (t))\d t \biggr]\\
& =G^{(k)}_{\al} f(x,k)+\E_{x,k} \biggl[ e^{-{\al} \tau}
G_{\al} f(X(\tau), \Lambda(\tau))\biggr]\\
& =G^{(k)}_{\al} f(x,k)+\sum_{l \in \ss\setminus \{k\}}\E_{x,k}
\biggl[
e^{-{\al} \tau} 
\biggl(-\,\frac{q_{kl}}{q_{kk}}\biggr)(X({\tau-}))
G_{\al} f(X({\tau-}),l)\biggr]\\
\ad =G^{(k)}_{\al} f(x,k)+\sum_{l \in \ss\setminus \{k\}}
G_{\al}^{(k)} (q_{kl} G_{\al} f(\cdot, l))(x),
\end{align*}
where  the last equality follows from
\eqref{eq:zeta}  in Lemma~\ref{lem-CP1}.
Hence we have
\begin{equation}\label{(FP18)}
G_{\al}f(x,k)=G^{(k)}_{\al}f(\cdot,k)(x)
+G^{(k)}_{\al}\biggl(\sum_{l \in \ss \setminus
\{k\}}q_{kl}G_{\al}f(\cdot,l)\biggr)(x).
\end{equation} Of course, we know that the second term on
the right hand side of (\ref{(FP18)}) equals
$$G^{(k)}_{\al}\biggl(\sum_{l \in \ss \setminus
\{k\}}q_{kl}G^{(l)}_{\al}f(\cdot,l)\biggr)(x) =
G^{(k)}_{\al}\biggl(\sum_{l \in \ss \setminus
\{k\}}q_{kl}G^{(l)}_{\al}\biggl(\sum_{l_{1} \in \ss \setminus
\{l\}}q_{ll_{1}}G_{\al}f(\cdot,l_{1})\biggr)\biggr)(x).$$ Hence, we
further obtain that for any fixed $k \in \ss$ and any integer $m\ge
1$,
\begin{equation}\label{(FP19)}
G_{\al}f(x,k)=\sum_{i=0}^{m} \psi^{(k)}_{i}(x)+R^{(k)}_{m}(x),
\end{equation} where
\begin{align*}
&  \psi^{(k)}_{0}=G^{(k)}_{\al}f(\cdot,k), \\
 & \psi^{(k)}_{1}=G^{(k)}_{\al}\biggl(\sum_{l \in \ss \setminus
\{k\}}q_{kl}G^{(l)}_{\al}f(\cdot,l)\biggr)=G^{(k)}_{\al}\biggl(\sum_{l
\in \ss \setminus \{k\}}q_{kl}\psi^{(l)}_{0}\biggr), \\\intertext{and for $ i \ge 1$,}
& \psi^{(k)}_{i}=G^{(k)}_{\al}\biggl(\sum_{l
\in \ss \setminus \{k\}}q_{kl}\psi^{(l)}_{i-1}\biggr).
\end{align*} 
By Assumption \ref{assump-Lhasonlyone}  we know that $+\infty>H:=\max \{\|q_{kk}\|: k
\in \ss\}\ge \max \{\|q_{kl}\|: k\ne l \in \ss\}$, where $\|q_{kl}\|$ denotes the uniform (or supremum) norm of the function $q_{kl}$ as before and constant $H$ is the same as that in Assumption~\ref{FP1}. Therefore,
$$\|\psi^{(k)}_{1}\| \le \sum_{l \in \ss \setminus
\{k\}}\|G^{(k)}_{\al}(q_{kl}\psi^{(l)}_{0})\|\le \frac{H}{\al}
\sum_{l \in \ss \setminus \{k\}}\|\psi^{(l)}_{0}\|.
$$ Thus, we get that
$$\sum_{k \in \ss}\|\psi^{(k)}_{1}\|\le \frac{(n_{0}-1)H}{\al}
\sum_{k \in \ss}\|\psi^{(k)}_{0}\|\le \frac{1}{2} \sum_{k \in
\ss}\|\psi^{(k)}_{0}\|$$ when $\al \ge \al_{1}:=2(n_{0}-1)H$. A similar
argument yields that for $i \ge 1$,
\begin{equation}\label{(FP20)}
\sum_{k \in \ss}\|\psi^{(k)}_{i}\|\le\frac{1}{2} \sum_{k \in
\ss}\|\psi^{(k)}_{i-1}\|\le \frac{1}{2^{i}} \sum_{k \in
\ss}\|\psi^{(k)}_{0}\|\end{equation} and
\begin{equation}\label{(FP21)}
\|R^{(k)}_{m}(\cdot)\|\le \frac{1}{2^{m}} \sum_{k \in
\ss}\|G_{\al}f(\cdot,k)\|\end{equation} when $\al \ge \al_{1}$.
Combining (\ref{(FP20)}) and (\ref{(FP21)}) with (\ref{(FP19)}) and
letting $m \uparrow \infty$, we conclude that for each $k \in \ss$,
$G_{\al}f(\cdot,k)=\sum_{i=0}^{\infty} \psi^{(k)}_{i}$, which
clearly implies (\ref{(FP17)}).  The lemma is proved. \end{proof}


 Lemma \ref{lem-resolvant-series} and in particular \eqref{(FP17)} establishes   the relationship between   the resolvent of $(X,\Lambda)$ and those of the killed L\'{e}vy type processes $X^{(k)}$, $k \in \ss$. Now we are in the position to give the proof of Theorem \ref{thm-Feller}.

\begin{proof}[Proof of Theorem \ref{thm-Feller}] Denote the transition probability
family of Markov process $(X,\Lambda)$ by $\{P(t,(x,k),A): t\ge 0,
(x,k)\in \rr^{d} \times \ss, A\in {\cal B}(\rr^{d} \times \ss)\}$. We first prove that for any given $t>0$, $x\in
\rr^{d}$, $k,l\in \ss$ and $A \in \B(\rr^{d})$,
\begin{align}  \nonumber P & (t,(x,k),A\times \{l\}) \\ \nonumber
 & =\delta_{kl} P^{(k)}(t,x,A)+\sum_{m=1}^{+\infty} \ \ \idotsint\limits_{0<t_{1}<t_{2}<\cdots
<t_{m}<t}
  \sum_{{l_{0}, l_{1}, l_{2}, \cdots, l_{m} \in
\ss}\atop{l_{i}\neq l_{i+1}, l_{0}=k, l_{m}=l}}\int_{\rr^{d}} \cdots
\int_{\rr^{d}}P^{(l_{0})}(t_{1},x,\d y_{1}) q_{l_{0}l_{1}}(y_{1})\\ \label{(FP22)}
&\ \quad \times  P^{(l_{1})}(t_{2}-t_{1},y_{1},\d y_{2})\cdots
q_{l_{m-1}l_{m}}(y_{m})P^{(l_{m})}(t-t_{m},y_{m},A) \d t_{1} \d t_{2}
\cdots \d t_{m},
\end{align}  where $\delta_{kl}$ is the Kronecker
symbol in $k$, $l$, which equals $1$ if $k=l$ and  $0$ if $k\neq
l$. To prove (\ref{(FP22)}), denote its the right-hand side by
$\wdt {P}(t,(x,k),A\times \{l\})$ for brevity. For any bounded function
$f(x,k)$  defined on $\rr^{d}\times \ss$ such that $f(\cdot,k)$ is
Lipschitz continuous for each $k\in \ss$,
we define \begin{align}
\label{eq-Pt-f}
   P_{t}f(x,k) :& =\E_{x,k}[f(X(t),\La(t))] =\sum_{l\in \ss} \int_{\R^{d}} f(y,l) P(t, (x,k), \d y \times \{ l\}),\\ \intertext{and}
 \label{eq-tilde-Pt-f}  \nonumber \wdt P_{t} f(x,k)   :& =  \sum_{l\in \ss} \int_{\R^{d}} f(y,l) \wdt P(t, (x,k), \d y \times \{ l\}) \\  \nonumber
    & = \sum_{l\in \ss} \Biggl[\int_{\R^{d}} \delta_{kl} f(y,l) P^{(k)}(t,x,\d y) +\sum_{m=1}^{+\infty} \ \ \idotsint\limits_{0<t_{1}<t_{2}<\cdots
<t_{m}<t}   \\  \nonumber & \qquad \sum_{{l_{0}, l_{1}, l_{2}, \cdots, l_{m} \in \ss}\atop{l_{i}\neq l_{i+1}, l_{0}=k, l_{m}=l} }  \int_{\rr^{d}} \cdots
\int_{\rr^{d}}P^{(l_{0})}(t_{1},x,\d y_{1}) q_{l_{0}l_{1}}(y_{1})   P^{(l_{1})}(t_{2}-t_{1},y_{1},\d y_{2})\cdots
\\ &\qquad \times  q_{l_{m-1}l_{m}}(y_{m})P^{(l_{m})}(t-t_{m},y_{m},\d y) \d t_{1} \d t_{2}
\cdots \d t_{m}\Biggr],
\end{align} Since the process $(X,\La)$ has right continuous sample paths, it follows from the continuity of $f$ and the bounded convergence theorem that the function $t \mapsto P_{t} f(x,k)$ is right continuous. Similarly for every $l\in \ss$ and  each $m =0 ,1, \dots$, every term on the right-hand side of \eqref{eq-tilde-Pt-f} is a right-continuous function in $t$. Moreover, using Assumption  \ref{assump-Lhasonlyone} and the boundedness of the function $f$, we can see that the series on the right-hand side of \eqref{eq-tilde-Pt-f}  is absolutely convergent. Therefore it follows that  the function $t \mapsto \wdt P_{t} f(x,k)$ is also right continuous.

On the other hand, using Lemma \ref{lem-resolvant-series} and in particular
\eqref{(FP17)}, for any $\alpha > 0$, we have \begin{displaymath}
\int_{0}^{\infty} e^{-(\alpha+ \alpha_{1}) t} e^{\alpha_{1} t }P_{t} f(x,k) \d t = \int_{0}^{\infty}e^{-(\alpha+ \alpha_{1}) t} e^{\alpha_{1} t }\wdt P_{t} f(x,k) \d t,
\end{displaymath}  where $\alpha_{1}$ is as in the statement of Lemma \ref{lem-resolvant-series}.
Since both $ P_{t} f(x,k)$ and $\wdt P_{t} f(x,k)$ are right continuous in $t$,  we can apply the uniqueness theorem of Laplace
transform (refer to \cite[Theorem 1.38]{Chen04}) to conclude that  that $ e^{\alpha_{1} t }P_{t} f(x,k) =  e^{\alpha_{1} t }\wdt P_{t} f(x,k)$. That is,
\begin{equation}\label{(FP23)}
\sum_{l \in \ss}\int f(y,l )P(t,(x,k),\d y\times
\{l \})=\sum_{l \in \ss}\int
f(y,l)\wdt {P}(t,(x,k),\d y\times \{l \})
\end{equation}

Now we prove \eqref{(FP22)} by the Monotone Class
Theorem (see, e.g., Theorem 1.35 in \cite{Chen04}). Denote  by $L$ the family of
 bounded and Borel measurable functions defined on $\R^{d}\times \ss$ such that \eqref{(FP23)}  holds.
From the above argument, we know that  $L$ contains all bounded and Lipschitz continuous functions on $\R^{d}\times \ss$.
Next we show  that $L$ is a
so-called $\mathcal{L}$-system (c.f. Definition 1.34 in Section 1.5 of \cite{Chen04}). Firstly, $L$ obviously contains the
constant function $1$. Secondly, for $c_{1}$ and $c_{2}$ in $\rr$
and $f_{1}$ and $f_{2}$ in $L$, we clearly have
$c_{1}f_{1}+c_{2}f_{2}$ in $L$. Thirdly, if  $f_{n}\in L$ with $0\le
f_{n}\uparrow f$, then $f \in L $ by the monotone convergence theorem.
Hence, according to the definition of $\mathcal{L}$-system
(\cite[Definition 1.34]{Chen04}), $L$ is an
$\mathcal{L}$-system.
Moreover, let ${\cal C}$ denote the set of all
the open sets in $\rr^{d} \times \ss$. Note that ${\cal C}$ is a
$\pi$-system and recall that $L$ contains the set of all bounded
Lipschitz continuous functions defined on $\rr^{d} \times \ss$.
Therefore, by virtue of the monotone class theorem (refer to
\cite[Theorem 1.35]{Chen04}), the family $L$ contains the set of all
bounded measurable functions defined on $\rr^{d} \times \ss$. In
particular, for any given $A \in \B(\rr^{d})$ and $l \in \ss$, the
family $L$  contains the function
$\one_{A\times \{l\}}(x,k)$, which implies that (\ref{(FP22)})
holds.

Finally, we use (\ref{(FP22)}) to prove the Feller property for
$(X,\Lambda)$. By Lemma~\ref{FP5}, we know that for every $k \in
\ss$, $X^{(k)}$ has the Feller property. Therefore, in view of
Proposition 6.1.1 in \cite{MeynT-93} and Assumption~\ref{assump-Lhasonlyone}, we
derive that $P^{(k)}(t,x,A)$ and every term in the series on the right-hand side of \eqref{(FP22)}
are lower semicontinuous with respect to $x$
whenever $A$ is an open set in $\B(\rr^{d})$. This then implies that
the left-hand side of (\ref{(FP22)}) is lower semicontinuous with
respect to $(x,k)$ for every $l \in \ss$ whenever $A$ is an open set
in $\B(\rr^{d})$ by noting that $\ss$ is a finite set and has
discrete metric. Consequently, $(X,\Lambda)$ has the Feller property
(see Proposition 6.1.1 in \cite{MeynT-93} again). The theorem is
proved. \end{proof}


\section{Strong Feller Property}\label{sect-str-Feller}
In this section, we study the strong Feller property for the coordinate process $(X(t),\La(t))$ in the underlying probability space $(\Omega, \F, \P)$ as specified in Section \ref{sect-Feller}. We first make the following assumption.

\begin{Assumption}\label{Assump-Uniform-Elliptic}
There exists a $\lambda_0 > 0$ such that $\langle \xi, a(x,k)\xi\rangle \ge \lambda_0|\xi|^2$ for all $x,\xi\in \R^d$ and $k\in \ss$.
Denote by $\sigma_{\lambda_{0}}(x,k)$ the unique symmetric nonnegative definite matrix-valued function such that $\sgla^{2} (x,k) = a(x,k)- \lambda_{0}I$.
In addition, there exist   positive constants  $\delta_{0}, H $ and a nonnegative function $\vartheta $ defined on $[0,\delta_{0}]$ satisfying $\lim_{r\to 0} \vartheta (r) =0$ such that
\begin{align}
\label{eq1-str-feller-condition}
 &  2  \lan x-z, b(x,k)-b(z,k) \ran + |\sgla(x,k)-\sgla(z,k)|^{2 }     \le 2H |x-z| \vartheta(|x-z|),\\
\label{eq2-str-feller-condition}
& \int_{\R^{d}_{0}}  |u| \|\nu(x,k, \cdot)- \nu(z,k,\cdot)\|(\d u)  \le \frac{ H}{2} \vartheta(|x-z|)
\end{align}  for all $x,z\in \R^{d}$ with $|x-z| \le \delta_{0}$
and all $k\in\ss$. 
\end{Assumption}

 \begin{Remark}\label{rem-about-uniform ellipticity} The uniform ellipticity condition for the diffusion matrix $a(x,k)$ in Assumption \ref{Assump-Uniform-Elliptic} is quite standard in the literature. Indeed,   similar assumptions were used in \cite{PriolaW-06, Qiao-14,P-Zabczyk-1995} to obtain the  strong Feller property.
 \end{Remark}

\begin{Proposition}\label{prop-str-Fe}
Under Assumptions \ref{FP1},  \ref{FP4}, and \ref{Assump-Uniform-Elliptic}, for
each $k \in \ss$,   both the L\'{e}vy type process $\wdt X^{(k)}$ and the killed L\'{e}vy type process $X^{(k)}$
are strong Feller.
\end{Proposition} 

 \begin{proof}
The proof is motivated by \cite{PriolaW-06}. Fix an arbitrary $k\in \ss$ throughout the proof. Let $\sigma_{\lambda_{0}}(x,k)$ be as in Assumption \ref{Assump-Uniform-Elliptic} 
and put $c(x,z,k) : = \lambda_0(I - 2 (x-z) (x-z)^T /|x-z|^{2})  + \sigma_{\lambda_{0}}(x,k)  \sigma_{\lambda_{0}}(z,k)^T $ for all $x,z\in \R^d$.
  For $x,z \in
\R^{d}$, set
$$\wdh a(x,z,k)=\left(\begin{array}{cc}
a(x,k) & c(x,z,k) \\
c(x,z,k)^{T}  & a(z,k)
\end{array}\right),\quad
b(x,z,k)=\left(\begin{array}{c}
b(x,k)\\
b(z,k) \end{array}\right).$$  We can verify directly that $\wdh a(x,z,k)$ is symmetric and nonnegative
definite for all $x,z \in \rr^{d}$. Then we define \begin{displaymath}
\wdh \Omega_{d}(k) h(x,z) := \frac{1}{2} \tr(\wdh a(x,z,k) \nabla^{2}  h(x,z)) + \langle b(x,z,k), \nabla h(x,z)\rangle,
\end{displaymath} and \begin{equation}\label{eq-L-hat-defn}
\wdh \LL_{k} h(x,z) : = \wdh \Omega_{d}(k) h(x,z) + \wdt \Omega_{j}(k) h(x,z),
\end{equation} where $h \in C_{0}^{2} (\R^{d}\times \R^{d})$ and $\wdt \Omega_{j}(k) $ is defined in \eqref{eq-Omega-j-defn}. Let \begin{align*}
 & A(x,z,k) =  a(x,k) + a(z,k) - 2 c(x,z,k),\\
    & \lbar A(x,z,k) = \frac{1}{|x-z|^{2}} \langle x-z, A(x,z,k) (x-z)\rangle,  \\
    &  B(x,z,k) = \langle x-z, b(x,k) - b(z,k)\rangle.
\end{align*} Straightforward computations lead to    \begin{align}\label{eq-tr(A(x,z,k))-estimate}
   \tr( A(x,z,k) ) =\|\sigma(x,k)-\sigma(z,k)\|^{2 } + 4 \lambda_{0}  \text{ and }
    \lbar A(x,z,k) \ge 4 \lambda_{0} .   \end{align}
Consider the function $F(r): = \frac{r}{1+r}, r\ge 0$. Then $F'(r) = \frac{1}{(1+r)^{2}} > 0$ and $F''(r) = \frac{-2}{(1+r)^{3}}  <  0$ for all $r \ge 0$.
 Consequently it follows from \eqref{eq1-str-feller-condition} and \eqref{eq-tr(A(x,z,k))-estimate} that
\begin{align}\label{eq2-strFe}
\nonumber \wdt \Omega_{d}(k) F (|x-z|) & = \frac{1}{2}F '' (|x-z|)\lbar A (x,z,k)  \\
 \nonumber   & \qquad + \frac{F '(|x-z|) }{2|x-z|} \big[ \tr(A(x,z,k))-\lbar A(x,z,k) + 2  B(x,z,k)\big] \\
 \nonumber   & \le 2 \lambda_{0}F '' (|x-z|) + H F '(|x-z|)\vartheta( |x-z|)\\
    & = \frac{-4\lambda_{0}}{(1+|x-z|)^{3}} +\frac{ H}{(1+|x-z|)^{2}} \vartheta (|x-z|),
\end{align}   for all $x,z\in \R^{d}$ with $|x-z|\le \delta_{0}$.

Next we estimate $ \wdt \Omega_{j}(k) F (|x-z|)$.  To this end, we note that since $F$ is concave, it follows that for any $x,z\in \R^{d}$ and $u\in \R^{d}_{0}$,  we have  \begin{align*}
  F&(|x+u-z|) - F(|x-z|)  - \lan\nabla_{x} F(|x-z|), u\ran {\mathbf{1}}_{B(0, \e_{0})}(u)\\
    & \le   F'(|x-z|) (|x+u-z|- |x-z| ) - \frac{F'(|x-z|)}{|x-z|} \lan x-z, u\ran {\mathbf{1}}_{B(0, \e_{0})}(u)\\
    &  \le \frac{2|u|}{(1+|x-z|)^{2}}.
\end{align*} Hence it follows that \begin{align*}
\int \big[ F(|x+u-z|) -   F(|x-z|)  -   \lan\nabla_{x} F(|x-z|), u\ran {\mathbf{1}}_{B(0, \e_{0})}(u)\big] \big(\nu(x,k,\d u) - \nu(z,k,\d u) \big)^{+} \\
    \le \frac{2}{(1+|x-z|)^{2}} \int |u| \big(\nu(x,k,\d u) - \nu(z,k,\d u) \big)^{+}.
\end{align*} In the same manner, we have
\begin{align*}
\int  \big[ F(|x -(z+u)|) - F(|x-z|)  - \lan\nabla_{z} F(|x-z|), u\ran {\mathbf{1}}_{B(0, \e_{0})}(u)\big] \big(\nu(x,k,\d u) - \nu(z,k,\d u) \big)^{+} \\
       \le \frac{2}{(1+|x-z|)^{2}} \int |u| \big(\nu(z,k,\d u) - \nu(x,k,\d u) \big)^{+}.
\end{align*}
On the other hand, since $\nabla_{x} F(|x-z|) = - \nabla_{z}F(|x-z|)$, we have
\begin{align*}
   \int \big[ &  F (|x+u-z-u|) - F (|x-z|) -\langle \nabla_{x} F (|x-z|), u \rangle {\mathbf{1}}_{B(0, \e_{0})}(u)\\  & -  \langle \nabla_{z} F (|x-z|), u \rangle {\mathbf{1}}_{B(0, \e_{0})}(u) \big] \big( \nu (x,k,\d u)  \wedge \nu (z,k,\d u) \big) = 0. \end{align*}
Then, using the definition of $\wdt \Omega_{j}(k)$ in \eqref{eq-Omega-j-defn} and condition \eqref{eq2-str-feller-condition}, we obtain
\begin{align}\label{eq3-strFe}
 \nonumber &  \wdt \Omega_{j}  (k) F(|x-z|)  \\  \nonumber&\  \le  \frac{2}{(1+|x-z|)^{2}} \biggl[\int |u| \big(\nu(x,k,\d u) - \nu(z,k,\d u) \big)^{+} +    \int |u| \big(\nu(z,k,\d u) - \nu(x,k,\d u) \big)^{+}  \biggr] \\
  \nonumber  &\    \le   \frac{2}{(1+|x-z|)^{2}} \int |u| \|\nu(x,k,\cdot ) - \nu(z,k,\cdot) \|(\d u)\\
    & \le \frac{H}{(1+|x-z|)^{2}} \vartheta(|x-z|),
\end{align}  for all $x,z\in \R^{d}$ with $|x-z|\le \delta_{0}$

Plugging \eqref{eq2-strFe} and \eqref{eq3-strFe} into \eqref{eq-L-hat-defn}, we obtain  that for all  $x,z\in \R^{d}$ with $|x-z|\le \delta_{0}$,
\begin{align*}
\wdh \LL_{k} F (|x-z|) &  \le \frac{-4\lambda_{0}}{(1+|x-z|)^{3}} +\frac{ 2H}{(1+|x-z|)^{2}} \vartheta (|x-z|)\\
& \le \frac{-4\lambda_{0}}{(1+ \delta_{0})^{3}} + 2 H \vartheta (|x-z|). \end{align*}
Furthermore, since $\lambda_{0} > 0$ and $\lim_{r \downarrow 0} \vartheta(r) =0$, it follows that there exist positive constants  $\kappa$ and $\delta $   ($0 < \delta < \delta_{0}$), we have
\begin{equation}
\label{eq-L-k F(|x-z|) estimate}
\wdh \LL_{k} F (|x-z|)  \le -\kappa, \text{ for all }0 < |x-z| \le \delta.
\end{equation}

Given $x\neq z$ with $\delta > |x-z | > \frac{1}{m_{0}}$, where $m_{0}\in \mathbb N$.  Let $(\wdt X^{(k) (x)},\wdt Z^{(k) (z)})$ be the coupling    process corresponding to the operator $\wdh \LL_{k}$ and denote by $T$ the coupling time. 
 For $n, N \in \mathbb N$ and the $  \delta $ in  \eqref{eq-L-k F(|x-z|) estimate}, define
\begin{align*}
 & T_n : =   \inf \Bigl\{  t \ge 0:   | \wdt X^{(k) (x)} (t) - \wdt Z^{(k) (z)} (t) | < \frac{1}{n}\Bigr\}, \\
    &  \sigma_{N} : =  \inf \{  t \ge 0:   | \wdt X^{(k) (x)} (t) | + |\wdt  Z^{(k) (z)} (t) | > N\}, \end{align*} and
    \begin{align*}
    & S_{\delta}: = \inf\{t \ge 0:  |\wdt  X^{(k) (x)} (t) - \wdt  Z^{(k) (z)} (t) | > \delta \}.
\end{align*}  We have
\begin{align*}
0 & \le F (\delta)\P_{k} \set{T_{n}\wedge \sigma_{N} > S_{\delta}}  \\ &  \le  \E_{k} [F ( | \wdt X^{(k) (x)} (T_{n}\wedge S_{\delta} \wedge \sigma_{N} ) - \wdt Z^{(k) (z)} (T_{n}\wedge S_{\delta} \wedge \sigma_{N}) |)] \\
& = F (|x-z|) + \E_{k} \biggl[ \int_{0}^{T_{n}\wedge S_{\delta} \wedge \sigma_{N}} \wdh \LL_{k} F (| \wdt X^{(k) (x)}-  \wdt Z^{(k) (z)}|) \d s\biggr]\\
& \le F (|x-z|) -  \kappa\E_{k}[T_{n}\wedge S_{\delta} \wedge \sigma_{N}],
\end{align*} where the last inequality follows from \eqref{eq-L-k F(|x-z|) estimate}.  Then it follows that
\begin{displaymath}
F (\delta)\P_{k} \set{T_{n}\wedge \sigma_{N} > S_{\delta}} +\kappa\, \E_{k}[T_{n}\wedge S_{\delta} \wedge \sigma_{N}] \le F (|x-z|).
\end{displaymath}Since $T_{n} \to T$ a.s. as $n \to \infty$ and $\sigma_{N} \to \infty$ a.s. as $N \to \infty$,  we have
\begin{displaymath}
F (\delta)\P_{k} \set{T  > S_{\delta}} +   \kappa\, \E_{k}[T \wedge S_{\delta} ] \le F (|x-z|).
\end{displaymath}
Then for any $t >0$ and $0 < |x-z| < \delta$,
\begin{align*}
 \P_{k} \set{T > t}    & = \P_{k}\set{T > t, S_{\delta} > t} + \P_{k}\set{T > t, S_{\delta} \le t}    \\
    &   \le  \P_{k}\set{T  \wedge S_{\delta} > t} + \P_{k}\set{T >  S_{\delta} }  \\
    & \le \frac{1}{t} \E_{k}[T  \wedge S_{\delta} ] +  \P_{k}\set{T >  S_{\delta} } \\
    & \le \biggl( \frac{1}{t \kappa} + \frac{1}{F (\delta)}\biggr)  F (|x-z|).
\end{align*}
  This implies the strong Feller property for the L\'{e}vy type process $\wdt X^{(k)}$ immediately. Indeed, for any $f \in \B_{b}(\R^{d})$,   $t>0$, and $0 < |x-z| < \delta$, we have
\begin{align*}
\bigl |   \E_{k}  \bigl[f(\wdt X^{(k)(x)}(t))\bigr]-\E_{k} \bigl[f(\wdt X^{(k)(z)}(t))\bigr]\bigr|  & \le \E_{k} \bigr[\bigl|f(\wdt X^{(k)(x)}(t))-f(\wdt X^{(k)(z)}(t)) \bigr|\bigr] \\
 & =  \E_{k} \bigr[\bigl|f(\wdt X^{(k)(x)}(t))-f(\wdt X^{(k)(z)}(t)) \bigr| I_{\{ T > t\}}\bigr] \\
 & \le 2 \|f\|_{\infty}  \P_{k}\set{  T> t}  \to 0, \text{ as }|x-z| \to 0.
\end{align*}  

Finally, as in the proof of Lemma \ref{FP5}, for any $f \in \B_{b}(\R^{d})$,   $t>0$, and $0 < |x-z| < \delta$, we  can write
\begin{align*}
  \bigl|  & \E_{k}  \bigl[f(X^{(k)(x)}(t))\bigr]-\E_{k} \bigl[f(X^{(k)(z)}(t))\bigr]\bigr|      \\
    & \le  \E_{k}  \biggl[ \Bigl| f(\wdt {X}^{(k)(x)}(t)) e^{\int_{0}^{t}q_{kk}(\wdt {X}^{(k)(x)}(s))\d s}
 - f(\wdt {Z}^{(k)(z)}(t))
 e^{  \int_{0}^{t}q_{kk}(\wdt {Z}^{(k)(z)}(s))\d s }\Bigr|\biggr] \\
 & \le \E_{k}  \Bigl[ \big| f(\wdt {X}^{(k)(x)}(t)) -  f(\wdt {Z}^{(k)(z)}(t))\big|  e^{\int_{0}^{t}q_{kk}(\wdt {X}^{(k)(x)}(s))\d s} \Bigr] \\
  & \quad+ \E_{k}\Bigl [f(\wdt {Z}^{(k)(z)}(t)) \big|e^{\int_{0}^{t}q_{kk}(\wdt {X}^{(k)(x)}(s))\d s} -e^{  \int_{0}^{t}q_{kk}(\wdt {Z}^{(k)(z)}(s))\d s } \big| \Bigr] \\
  & \le 2 \| f\|_{\infty} \P_{k}\set{  T> t} +  \|f \|_{\infty} \E_{k}\biggl[\biggl|\int_{0}^{t}q_{kk}(\wdt {X}^{(k)(x)}(s))\d s-  \int_{0}^{t}q_{kk}(\wdt {Z}^{(k)(z)}(s))\d s\biggr|\biggr] \\
  & \le 2 \| f\|_{\infty}  \biggl( \frac{1}{t\kappa} + \frac{1}{F (\delta)}\biggr)  F (|x-z|) + H (n_{0}-1) \|f \|_{\infty}\E_{k}\biggl[ \int_{0}^{t} \big|\wdt X^{(k)(x)}(s) - \wdt Z^{(k)(z)}(s) \big| \d s\biggr] \\
  & \le  2 \| f\|_{\infty} \biggl( \frac{1}{t\kappa} + \frac{1}{F (\delta)}\biggr)  F (|x-z|)  +  H (n_{0}-1)  \|f \|_{\infty}  \int_{0}^{t} G^{-1}(G(| x-z |) + 3H  s) \d s,
\end{align*} where the second last inequality above follows from Assumption \ref{FP4} and the last inequality follows from \eqref{(FP11)}.
Note
 that $F (\cdot)$  is continuous with $F (0 ) =0$.  In addition, recall that we argued in the proof of Lemma \ref{FP5} that $\int_{0}^{t} G^{-1}(G(| x-z |) + 3H  s) \d s \to 0$ as $|x-z| \to 0$. Thus it follows that $ \bigl|    \E_{k}  \bigl[f(X^{(k)(x)}(t))\bigr]-\E_{k} \bigl[f(X^{(k)(z)}(t))\bigr]\bigr| \to 0 $ as $|x-z| \to 0$.
On the other hand, for any $|x-z| \ge \delta$, we have
\begin{displaymath}
  \bigl|    \E_{k}  \bigl[f(X^{(k)(x)}(t))\bigr]-\E_{k} \bigl[f(X^{(k)(z)}(t))\bigr]\bigr|  \le 2 \delta^{-1}\|f\|_{\infty} |x-z|.
\end{displaymath}
 Therefore we obtain the desired  strong Feller property for the killed L\'evy process $X^{(k)}$. This completes the proof.
\end{proof}

With Proposition  \ref{prop-str-Fe} at our hands, we can use exactly the same arguments as those in the  proof of Theorem \ref{thm-Feller} to establish the following theorem.
\begin{Theorem}\label{thm-str-Feller}
Under the conditions of Proposition  \ref{prop-str-Fe}, the process $(X,\La) $ possesses the strong Feller property.
\end{Theorem}

\section*{\bf\large Acknowledgements}
The research was supported in
 part by the National Natural Science Foundation of China under Grant No. 11671034, the Beijing Natural Science Foundation under Grant No. 1172001,   the Simons foundation collaboration  under Grant No. 523736, and a grant from the Research Growth Initiative of UW-Milwaukee.

 \bibliographystyle{apalike}

\def\cprime{$'$}

\end{document}